\documentclass[12pt]{amsart}

\usepackage[colorlinks=true, pdfstartview=FitV, linkcolor=blue, citecolor=blue, urlcolor=blue, breaklinks=true]{hyperref}
\usepackage{amsmath,amsfonts,amssymb,amsthm,amscd,comment,paralist,stmaryrd,etoolbox,mathtools}
\usepackage[usenames,dvipsnames]{color}
\usepackage{mdwlist}

\newcommand\Z{\mathbb{Z}}
\newcommand\Q{\mathbb{Q}}

\newcommand\N{\mathbb{N}}

\newcommand\kk{\Bbbk}

\newcommand\id{\mathrm{Id}}

\newcommand\fh{\mathfrak{h}}

\newcommand\cF{\mathcal{F}}

\newcommand\cP{\mathcal{P}}

\newcommand\blambda{{\boldsymbol\lambda}}
\newcommand\bmu{{\boldsymbol\mu}}
\newcommand\bnu{{\boldsymbol\nu}}

\newcommand\Sy{\mathrm{Sym}}

\newcommand\prt{\mathrm{part}}
\newcommand\clr{\mathrm{color}}

\newcommand\pR[1]{\prescript{R}{}{#1}}
\newcommand\pL[1]{\prescript{L}{}{#1}}

\newcommand\qbin[2]{{#1 \atopwithdelims[] #2}}

\newcommand\ts{}

\DeclareMathOperator{\End}{End}

\DeclareMathOperator{\Ind}{Ind}

\newtheorem{theo}{Theorem}[section]
\newtheorem{prop}[theo]{Proposition}
\newtheorem{lem}[theo]{Lemma}
\newtheorem{cor}[theo]{Corollary}

\theoremstyle{definition}
\newtheorem{defin}[theo]{Definition}
\newtheorem{rem}[theo]{Remark}

\numberwithin{equation}{section}
\allowdisplaybreaks

\newtoggle{comments}    
\newtoggle{details}     
\newtoggle{prelimnote}  
\newtoggle{detailsnote} 

\toggletrue{detailsnote}

\iftoggle{comments}{%
  \newcommand{\comments}[1]{
    \begin{center}
      \parbox{6.5 in}{
        \color{red}
          {\footnotesize \textbf{Comments:} #1}
        \color{black}}
    \end{center}}
}{%
  \newcommand{\comments}[1]{}
}

\iftoggle{details}{%
  \newcommand{\details}[1]{
      \ \\
      \color{OliveGreen}
        {\footnotesize \textbf{Details:} #1}
      \color{black}
      \\
  }
}{%
  \newcommand{\details}[1]{}
}

\iftoggle{prelimnote}{%
  \newcommand{\prelim}{\textsc{Preliminary version} \bigskip}
}{%
  \newcommand{\prelim}{}
}

\begin{document}
%

\title{Twisted Heisenberg doubles}

\author{Daniele Rosso}
\address{D.~Rosso: Department of Mathematics and Statistics, University of Ottawa, and Centre de Recherches Math\'ematiques, Montr\'eal}
\urladdr{\url{http://mysite.science.uottawa.ca/drosso/}}
\email{drosso@uottawa.ca}

\author{Alistair Savage}
\address{A.~Savage: Department of Mathematics and Statistics, University of Ottawa}
\urladdr{\url{http://mysite.science.uottawa.ca/asavag2/}}
\email{alistair.savage@uottawa.ca}
\thanks{The second author was supported by a Discovery Grant from the Natural Sciences and Engineering Research Council of Canada.  The first author was supported by the Centre de Recherches Math\'ematiques and the Discovery Grant of the second author.}

\begin{abstract}
  We introduce a twisted version of the Heisenberg double, constructed from a twisted Hopf algebra and a twisted pairing.  We state a Stone--von Neumann type theorem for a natural Fock space representation of this twisted Heisenberg double and deduce the effect on the algebra of shifting the product and coproduct of the original twisted Hopf algebra.  We conclude by showing that the quantum Weyl algebra, quantum Heisenberg algebras, and lattice Heisenberg algebras are all examples of the general construction.
\end{abstract}


\prelim

\maketitle
\thispagestyle{empty}

\tableofcontents

%
\section{Introduction}
%

The \emph{Heisenberg double} is a generalization of the Heisenberg algebra.  One can form the Heisenberg double of any nonnegatively graded Hopf algebra satisfying some mild assumptions.  As a $\kk$-module, the Heisenberg double of a Hopf algebra $H^+$ over a commutative ring $\kk$ with dual $H^-$ is isomorphic to $H^+ \otimes_\kk H^-$, and the factors $H^+$ and $H^-$ are subalgebras.  The relations between elements of $H^+$ and elements of $H^-$ arise from the left regular action of $H^-$ on $H^+$.  The Heisenberg double of the Hopf algebra of symmetric functions is precisely the infinite-dimensional Heisenberg algebra.  Just like the Heisenberg algebra, the more general Heisenberg double also has a natural Fock space representation and a Stone--von Neumann type theorem (see~\cite[Th.~2.11]{SY13}).

In the current paper, we define a twisted version of the Heisenberg double.  In particular, we replace Hopf algebras by \emph{twisted Hopf algebras} and replace the Hopf pairing between $H^+$ and $H^-$ (identifying each as the dual of the other) by a \emph{twisted pairing}.  Provided this data satisfies a certain compatibility condition (see Definition~\ref{def:compatible-pair}), we define the associated \emph{twisted Heisenberg double}.  In the case that the twistings are all trivial, our definition reduces to the usual Heisenberg double.  It turns out that the twisted Heisenberg double also has a natural Fock space representation that satisfies a Stone--von Neumann type theorem (Theorem~\ref{theo:Fock-space-properties}).

The main motivation for our definition of the twisted Heisenberg double comes from the theory of categorification.  The infinite-dimensional Heisenberg algebra was conjecturally categorified in~\cite{Kho10} using the tower of symmetric groups, and then in~\cite{LS13} using the tower of Hecke algebras of type $A$.  Motivated by these and other constructions, a general approach to the categorification of the Heisenberg double was taken in~\cite{SY13}.  This general approach was in terms of towers of algebras and categories of modules over such towers.  On the other hand, many constructions in categorification involve categories of \emph{graded} modules.  Extending the work of~\cite{SY13} to the setting of categories of graded (super)algebras necessitates the introduction of a twisted Heisenberg double.  We refer the reader to~\cite{RS14} for the details of this extension.

The relationship to categorification raises natural questions about the twisted Heisenberg double.  For example, it is common in the categorification literature to consider grading shifts of certain functors such as induction and restriction.  Since induction and restriction correspond to multiplication and comultiplication in the twisted Heisenberg double, it is important to know how the corresponding shifts of these operations affect the twisted Heisenberg double.  We address this question in Section~\ref{sec:shift}.  Namely, we deduce the precise effect of such shifts and show that certain shifts do not change the isomorphism type of the twisted Heisenberg double (Theorem~\ref{theo:alpha-invariant}).

In the final three sections of the paper, we illustrate our construction with several examples, all motivated by the categorification literature.  First, in Section~\ref{sec:quantum-Weyl}, we show that the quantum Weyl algebra is a twisted Heisenberg double.  This is related to the fact that this algebra is categorified by categories of graded modules for the tower of nilcoxeter algebras (see~\cite[\S8]{RS14}).  Then, in Sections~\ref{sec:q-Heis} and~\ref{sec:lattice-Heis}, we consider the quantum Heisenberg algebras and lattice Heisenberg algebras related to various categorical constructions, such as those of~\cite{CL12,FJW00a,FJW00b,FJW02}.  It turns out that, even though these algebras are categorified by categories of graded modules,  they are, in fact, untwisted Heisenberg doubles.  By our general theory, one could then conclude that some of the grading shifts appearing in the categorical constructions do not affect the isomorphism class of the algebra being categorified.

Note that Hopf algebras can be viewed as Hopf algebra objects in the symmetric monoidal category of vector spaces.  Generalizing this viewpoint, one can view certain twisted Hopf algebras as Hopf algebra objects in the braided monoidal category of graded vector spaces.  In this approach, the twisting depends on either the inner or outer terms in a fourfold tensor product (see~\eqref{eq:twisted-mult}).  The twisted Hopf algebras in the current paper are more general, allowing a twisting depending on \emph{both} the inner and outer terms.  While in many specific examples it is possible to use a twisting depending on only one or the other, the more general approach seems more natural from various points of view.  For example, even if one considers a Hopf algebra whose twisting depends only on inner (or outer) terms, its twisted dual has a twisting that, a priori, depends on both inner \emph{and} outer terms (see Lemma~\ref{lem:dual-twisted-Hopf}).

\subsection*{Notation}

We let $\N$ and $\N_+$ denote the set of nonnegative and positive integers respectively.  We let $\kk$ be a commutative ring with unit.

\iftoggle{detailsnote}{
\medskip

\paragraph{\textbf{Note on the arXiv version}} For the interested reader, the tex file of the arXiv version of this paper includes hidden details of some straightforward computations and arguments that are omitted in the pdf file.  These details can be displayed by switching the \texttt{details} toggle to true in the tex file and recompiling.
}{}

\subsection*{Acknowledgements}

The authors would like to thank A.~Licata, J.~Sussan, and O.~Yacobi for helpful conversations.

%
\section{Dual twisted Hopf algebras} \label{sec:dual-hopf}
%

We fix a commutative ring $\kk$ with unit and all algebras, coalgebras, bialgebras and Hopf algebras will be over $\kk$.  We will denote the multiplication, comultiplication, unit, counit and antipode of a Hopf algebra by $\nabla$, $\Delta$, $\eta$, $\varepsilon$, and $S$ respectively.  We write the product $\nabla$ as juxtaposition when this will not cause confusion, and use Sweedler notation
\[ \ts
  \Delta(a) = \sum_{(a)} a_{1} \otimes a_{2}
\]
for coproducts.  For a $\kk$-module $V$, we will simply write $\End V$ for $\End_\kk V$.  All tensor products are over $\kk$ unless otherwise indicated.

Let $(\Lambda,+)$ be a commutative monoid isomorphic (as a monoid) to $\N^r$.  We denote the identity element of $\Lambda$ by $0$.   We say that the algebra $(H,\nabla,\eta)$ with multiplication $\nabla$ and unit $\eta$ is \emph{$\Lambda$-graded} if $H = \bigoplus_{\lambda \in \Lambda} H_\lambda$, where each $H_\lambda$, $\lambda \in \Lambda$, is finitely generated and free as a $\kk$-module, and
\[
  \eta(\kk) \subseteq H_0,\quad
  \nabla(H_\lambda \otimes H_\mu) \subseteq H_{\lambda + \mu},\quad \lambda,\mu \in \Lambda.
\]
Similarly, a coalgebra $(H,\Delta,\varepsilon)$ with comultiplication $\Delta$ and counit $\varepsilon$ is \emph{$\Lambda$-graded} if $H = \bigoplus_{\lambda \in \Lambda} H_\lambda$, where each $H_\lambda$, $\lambda \in \Lambda$, is finitely generated and free as a $\kk$-module, and
\begin{gather*}
  \varepsilon(H_\lambda) = 0 \text{ for } \lambda \in \Lambda\setminus\{0\}, \\
  \Delta(H_\lambda) \subseteq \bigoplus_{\mu+\nu=\lambda} H_\mu \otimes H_\nu,\quad \lambda,\mu,\nu \in \Lambda.
\end{gather*}

Fix $q \in \kk^\times$, and let $\chi=(\chi',\chi'')$ be a pair of biadditive maps $\chi',\chi'' \colon \Lambda\times\Lambda\to\Z$. Following \cite[Part~II.2]{Ri96}, define a new multiplication $*_\chi$ on $H \otimes H$ by the condition that, for homogeneous elements $a_i, b_i \in H$, $i=1,2$, we have
\begin{equation} \label{eq:twisted-mult}
  (a_1\otimes a_2) *_\chi (b_1\otimes b_2) = q^{\chi'(|a_2|,|b_1|)+\chi''(|a_1|,|b_2|)}a_1b_1\otimes a_2b_2,
\end{equation}
where $|a|$ denotes the degree of a homogeneous element $a \in H$.  (Whenever we write an expression involving $|a|$ for some $a \in H$, we implicitly assume that $a$ is homogeneous.)  Notice that this is similar to the definition of the product in~\cite[p.~3]{L10}.  Since $\chi',\chi''$ are biadditive, $*_\chi$ is associative, and we denote by $(H\otimes H)_\chi$ this \emph{twisted} associative algebra structure.

\begin{lem} \label{lem:twisting-mult-or-comult}
  If $(H,\nabla,\varepsilon)$ is a $\Lambda$-graded algebra and $(H,\Delta,\eta)$ is a $\Lambda$-graded coalgebra, then $\Delta \colon H \to (H \otimes H)_\chi$ is an algebra morphism if and only if $\nabla \colon (H \otimes H)_\chi \to H$ is a coalgebra morphism, where the coalgebra structure on $(H \otimes H)_\chi$ is given by
  \[
    (\Delta \otimes \Delta)_\chi (a \otimes b) = \sum_{(a),(b)} q^{\chi'(|a_2|,|b_1|)+\chi''(|a_1|,|b_2|)} a_1 \otimes b_1 \otimes a_2 \otimes b_2.
  \]
\end{lem}

\begin{proof}
  The proof in the case that $\chi''=0$ can be found in~\cite[Lem.~2.7]{LZ00}.  The general result can be seen as follows.  The map $\Delta \colon H \to (H \otimes H)_\chi$ is an algebra homomorphism if and only if, for all $a,b \in H$, we have
  \[
    \Delta(ab) = \Delta(a) *_\chi \Delta(b)
    = \sum_{(a),(b)} q^{\chi'(|a_2|,|b_1|) + \chi''(|a_1|,|b_2|)} a_1 b_1 \otimes a_2 b_2
    = (\nabla \otimes \nabla) (\Delta \otimes \Delta)_\chi (a \otimes b),
  \]
  which is precisely the statement that $\nabla \colon (H \otimes H)_\chi \to H$ is a coalgebra homomorphism.
\end{proof}

\begin{defin}[$\Lambda$-graded connected $(q,\chi)$-bialgebra]\label{Lambda-grad}
  Suppose $(H,\nabla,\varepsilon)$ is a $\Lambda$-graded algebra and $(H,\Delta,\eta)$ is a $\Lambda$-graded coalgebra.  We say that $(H,\nabla,\Delta,\varepsilon,\eta)$ is a \emph{$\Lambda$-graded connected twisted bialgebra}, or, more precisely, a \emph{$(q,\chi)$-bialgebra} if $H_0 = \kk 1_H$ and
  \[
    \Delta \colon H \to (H\otimes H)_\chi
  \]
  is an algebra homomorphism.  It is a \emph{$(q,\chi)$-Hopf algebra} if, in addition, there is a $\kk$-linear map $S \colon H \to H$, called the \emph{antipode}, such that
  \[
    \nabla (\id \otimes S) \Delta = \eta \varepsilon = \nabla(S \otimes \id) \Delta.
  \]
  We say that $H$ is a \emph{twisted Hopf algebra} if it is a $(q,\chi)$-Hopf algebra for some choice of $(q,\chi)$.  We will write $(q,\chi',\chi'')$ instead of $(q,(\chi',\chi''))$ when we wish to make the components of $\chi$ explicit.
\end{defin}

In fact, a $\Lambda$-graded connected twisted bialgebra is always a twisted Hopf algebra.  The proof of the existence of the antipode in the case $\chi''=0$ can be found in~\cite[Th.~2.10]{LZ00}.  The proof is identical in the more general setting.  In particular, the antipode $S$ is defined in the same way (see the proof of Lemma~\ref{lem:S-adjointness}).

\begin{rem}
  In the case where $q^{\chi'(\lambda_1,\lambda_2)+\chi''(\mu_1,\mu_2)}=1$ for all $\lambda_1,\lambda_2, \mu_1,\mu_2\in\Lambda$ (for example when $q=1$ or $\chi' = \chi'' = 0$), we have that $(H\otimes H)_\chi=H\otimes H$ with componentwise multiplication.  In this case, we recover the usual definitions of bialgebra and Hopf algebra.
\end{rem}

For a biadditive map $\zeta \colon \Lambda \times \Lambda \to \Z$, define
\[
  \zeta^T \colon \Lambda \times \Lambda \to \Z,\quad \zeta^T(\lambda,\mu)=\zeta(\mu,\lambda),\quad \lambda,\mu \in \Lambda.
\]

\begin{rem} \label{rem:non-unique}
  For a given twisted Hopf algebra $H$, the data $(q,\chi',\chi'')$ is not unique. Obviously if $H$ is a $(q,\chi',\chi'')$-Hopf algebra, it is also a $(q^{-1},-\chi',-\chi'')$-Hopf algebra. But even fixing the choice of $q$ does not determine $\chi$. For example if $H$ is a commutative $(q,\chi',\chi'')$-Hopf algebra, it is straightforward to verify that is is also a $(q,(\chi'')^T,(\chi')^T)$-Hopf algebra.
  \details{
    We have that
    \[
      \Delta(ab)=\Delta(a) *_\chi \Delta(b)=\sum_{(a),(b)} q^{\chi'(|a_2|,|b_1|)+\chi''(|a_1|,|b_2|)}a_1b_1\otimes a_2b_2,
    \]
    but also
    \begin{align*}
      \Delta(ab)&=\Delta(ba) =\Delta(b) *_\chi \Delta(a) \\
      &=\sum_{(a),(b)} q^{\chi'(|b_2|,|a_1|)+\chi''(|b_1|,|a_2|)}b_1a_1\otimes b_2a_2 \\
      &= \sum_{(a),(b)} q^{\chi'(|b_2|,|a_1|)+\chi''(|b_1|,|a_2|)}a_1b_1\otimes a_2b_2.
    \end{align*}
    Hence $H$ is also a $(q,(\chi'')^T,(\chi')^T)$-Hopf algebra.
  }
  Similarly, if $H$ is a cocommutative $(q,\chi',\chi'')$-Hopf algebra, then it is also a $(q,\chi'',\chi')$-Hopf algebra.
  \details{
    We have
    \begin{align*}
      \Delta(ab) &= \Delta^\text{op}(ab) = \text{flip} \circ \Delta(ab) \\
      &= \text{flip} \left( \Delta(a) *_\chi \Delta(b) \right) \\
      &= \text{flip} \left( \Delta^\text{op}(a) *_\chi \Delta^\text{op}(b) \right) \\
      &= \text{flip} \left( \sum_{(a),(b)} q^{\chi'(|a_1|,|b_2|) + \chi''(|a_2|, |b_1|)} a_2 b_2 \otimes a_1 b_1 \right) \\
      &= \sum_{(a),(b)} q^{\chi'(|a_1|,|b_2|) + \chi''(|a_2|, |b_1|)} a_1 b_1 \otimes a_2 b_2
    \end{align*}
  }
\end{rem}

\begin{defin}[Twisted pairing]\label{def:hopf-pairing}
  Suppose $H$ and $H'$ both have algebra and coalgebra structures, $c$ is an invertible element in $\kk$, and $\gamma = (\gamma', \gamma'')$ is a pair of biadditive maps $\gamma', \gamma'' \colon \Lambda \times \Lambda \to \Z$.  Then a \emph{$(c,\gamma)$-twisted pairing} is a bilinear map $\langle -, - \rangle \colon H \times H' \to \kk$ such that $\langle -,-\rangle |_{H_\lambda\times H'_\mu}  \equiv 0$ when $\lambda, \mu \in \Lambda$, $\lambda\neq\mu$, and
  \begin{gather*} \ts
    \langle xy, a \rangle = c^{\gamma'(|x|,|y|)} \langle x \otimes y, \Delta(a) \rangle, \\ \ts
    \langle x, ab \rangle = c^{\gamma''(|a|,|b|)}\langle \Delta(x), a \otimes b \rangle, \\
    \langle 1_H, a \rangle = \varepsilon(a),\quad \langle x, 1_{H'} \rangle = \varepsilon(x),
  \end{gather*}
  for all homogeneous $x,y \in H$, $a,b \in H'$, where we define
  \[
    \langle -, - \rangle \colon (H \otimes H) \otimes (H' \otimes H') \to \kk,\quad \langle x \otimes y, a \otimes b \rangle = \langle x, a \rangle \langle y, b \rangle,\quad x,y \in H,\ a,b \in H'.
  \]
  We will write $(c,\gamma',\gamma'')$ instead of $(c,(\gamma',\gamma''))$ when we wish to make the components of $\gamma$ explicit.
\end{defin}

\begin{lem} \label{lem:S-adjointness}
  Suppose that $H$ and $H'$ are twisted Hopf algebras and $\langle -, - \rangle \colon H \otimes H' \to \kk$ is a $(c,\gamma)$-twisted pairing.  If $\gamma' = \gamma''$, then $\langle x, S(a) \rangle = \langle S(x), a \rangle$ for all $x \in H$ and $a \in H'$.
\end{lem}

\begin{proof}
  The antipode $S$ of $H^+$ is defined recursively as follows (see~\cite[Th.~2.10]{LZ00}).  We have $S(1)=1$ and, for homogeneous $a \in H^+$, we can write $\Delta(a) = a \otimes 1 + 1 \otimes a + \sum a' \otimes a''$, where the sum is over a set of pairs $(a',a'') \in H^+$ where both $a'$ and $a''$ are of strictly positive degree.  Then
  \[ \ts
    S(a) = - a - \sum a'S(a'').
  \]
  The antipode of $H^-$ is defined in an analogous manner.

  Let $a \in H^+$ and $x \in H^-$.  Since the antipode preserves degree, and elements of different degree are orthogonal, we may assume that $a$ and $x$ are homogeneous of the same degree.  The result is clearly true if they are of degree zero.  Thus, assume the degree of $a$ and $x$ is positive and that the result holds for all elements of smaller degree.  Then
  \begin{align*}
    \langle x, S(a) \rangle &= \ts -\langle x, a \rangle - \sum \langle x, a' S(a'') \rangle \\
    &= \ts -\langle x, a \rangle - \sum c^{\gamma''(|a'|,|a''|)} \langle \Delta(x), a' \otimes S(a'') \rangle \\
    &= \ts -\langle x, a \rangle - \sum \sum c^{\gamma''(|a'|,|a''|)} \langle x' \otimes x'', a' \otimes S(a'') \rangle \\
    &= \ts -\langle x, a \rangle - \sum \sum c^{\gamma''(|x'|,|x''|)} \langle x' \otimes S(x''), a' \otimes a'' \rangle \\
    &= \ts -\langle x, a \rangle - \sum c^{\gamma'(|x'|,|x''|)} \langle x' \otimes S(x''), \Delta(a) \rangle \\
    &= \ts -\langle x, a \rangle - \sum \langle x' S(x''), a \rangle \\
    &= \langle S(x), a \rangle. \qedhere
  \end{align*}
\end{proof}

Recall that a bilinear map $\langle - , - \rangle \colon V \otimes W \to \kk$ is called a \emph{perfect pairing} if the induced map $\Phi \colon V \to W^*$ given by $\Phi(v)(w) = \langle v, w \rangle$ is an isomorphism.

\begin{defin}[Dual pair] \label{def:dual-pair}
  We say that $(H^+, H^-)$ is a \emph{$(c,\gamma)$-dual pair} of twisted Hopf algebras if $H^+$ and $H^-$ are both twisted Hopf algebras, and there exists a $(c,\gamma)$-twisted pairing $\langle -, - \rangle \colon H^- \times H^+ \to \kk$ such that $\langle -,- \rangle|_{H^{-}_\lambda\times H^+_\lambda}$, $\lambda \in \Lambda$, is a perfect pairing.  We say that the pair $(H^+,H^-)$ is \emph{twisted dual} if it is a $(c,\gamma)$-dual pair for some $(c,\gamma)$.
\end{defin}

\begin{lem} \label{lem:dual-twisted-Hopf}
  Suppose $H^+$ is a $(q,\chi)$-Hopf algebra, $H^-$ is both an algebra and a coalgebra, and $\langle -, - \rangle \colon H^- \times H^+ \to \kk$ is a $(q,\gamma)$-twisted pairing such that $\langle -,- \rangle|_{H^{-}_\lambda\times H^+_\lambda}$, $\lambda \in \Lambda$, is a perfect pairing.  Then $H^-$ is a $(q,\xi$)-Hopf algebra, where
  \begin{equation} \label{eq:xi-dual}
    \xi = (\xi',\xi''),\quad \xi' = (\chi')^T + \gamma' - (\gamma'')^T,\quad \xi'' = \chi'' + \gamma' - \gamma''.
  \end{equation}
  In particular, $(H^+,H^-)$ is a $(q,\gamma)$-dual pair of twisted Hopf algebras.
\end{lem}

\begin{proof}
  Let $x,y \in H^-$ and $a,b \in H^+$.  Then we have
  \begin{align*}
    \langle& \Delta(xy), a \otimes b \rangle = q^{-\gamma''(|a|,|b|)} \langle xy, ab \rangle \\
    &= q^{\gamma'(|x|,|y|)-\gamma''(|a|,|b|)} \langle x \otimes y, \Delta(ab) \rangle \\
    &= q^{\gamma'(|x|,|y|)-\gamma''(|a|,|b|)} \left\langle x \otimes y, \Delta(a) *_\chi \Delta(b) \right\rangle \\
    &= \ts q^{\gamma'(|x|,|y|)-\gamma''(|a|,|b|)} \left\langle x \otimes y, \sum_{(a),(b)} q^{\chi'(|a_2|,|b_1|) + \chi''(|a_1|,|b_2|)} a_1 b_1 \otimes a_2 b_2 \right\rangle  \\
    &= \ts q^{\gamma'(|x|,|y|)-\gamma''(|a|,|b|)} \sum_{(a),(b)} q^{\chi'(|a_2|,|b_1|) + \chi''(|a_1|,|b_2|)} \langle x, a_1b_1 \rangle \langle y, a_2b_2 \rangle \\
    &= \ts \sum_{(a),(b)} q^{\chi'(|a_2|,|b_1|) + \chi''(|a_1|,|b_2|)} q^{\gamma'(|x|,|y|)-\gamma''(|a|,|b|)+\gamma''(|a_1|,|b_1|) + \gamma''(|a_2|,|b_2|)} \langle \Delta(x), a_1 \otimes b_1 \rangle \langle \Delta(y), a_2 \otimes b_2 \rangle \\
    &= \ts \sum_{(a),(b),(x),(y)} q^{\chi'(|a_2|,|b_1|) + \chi''(|a_1|,|b_2|)} q^{\gamma'(|x|,|y|) - \gamma''(|a_1|,|b_2|) - \gamma''(|a_2|,|b_1|)} \langle x_1, a_1 \rangle \langle x_2, b_1 \rangle \langle y_1, a_2 \rangle \langle y_2, b_2 \rangle \\
    &= \ts \sum_{(x),(y)} q^{\chi'(|y_1|,|x_2|) + \chi''(|x_1|,|y_2|)} q^{\gamma'(|x|,|y|) - \gamma''(|x_1|,|y_2|) - \gamma''(|y_1|,|x_2|)} \langle x_1 \otimes y_1, \Delta(a) \rangle \langle x_2 \otimes y_2, \Delta(b) \rangle \\
    &= \ts \sum_{(x),(y)} q^{\chi'(|y_1|,|x_2|) + \chi''(|x_1|,|y_2|)} q^{\gamma'(|x|,|y|) - \gamma'(|x_1|,|y_1|) - \gamma'(|x_2|,|y_2|) - \gamma''(|x_1|,|y_2|) - \gamma''(|y_1|,|x_2|)} \langle x_1 y_1, a \rangle \langle x_2 y_2, b \rangle  \\
    &= \ts \sum_{(x),(y)} q^{\chi'(|y_1|,|x_2|) + \chi''(|x_1|,|y_2|)} q^{\gamma'(|x_1|,|y_2|) + \gamma'(|x_2|,|y_1|) - \gamma''(|x_1|,|y_2|) - \gamma''(|y_1|,|x_2|)} \langle x_1 y_1 \otimes x_2 y_2, a \otimes b \rangle \\
    &= \ts \left\langle \sum_{(x),(y)} q^{\chi'(|y_1|,|x_2|) + \chi''(|x_1|,|y_2|)} q^{\gamma'(|x_1|,|y_2|) + \gamma'(|x_2|,|y_1|) - \gamma''(|x_1|,|y_2|) - \gamma''(|y_1|,|x_2|)} x_1 y_1 \otimes x_2 y_2, a \otimes b \right\rangle \\
    &= \left\langle \Delta(x) *_{((\chi')^T + \gamma' - (\gamma'')^T, \chi'' + \gamma' - \gamma'')} \Delta(y), a \otimes b \right\rangle.
  \end{align*}
  The result then follows from the nondegeneracy of the bilinear form.
\end{proof}

%
\section{The twisted Heisenberg double} \label{sec:h-definition}
%

For the remainder of this section, we fix a $(q,\gamma)$-dual pair $(H^+,H^-)$ of twisted Hopf algebras, where $H^{+}$ is a $(q,\chi)$-Hopf algebra and $H^-$ is a $(q,\xi)$-Hopf algebra, with $\xi$ given by~\eqref{eq:xi-dual}.

Any $a \in H^+$ defines an element $\pL{a} \in \End H^+$ by left multiplication.  Similarly, any $x \in H^-$ defines an element $\pR{x} \in \End H^- $ by right multiplication, whose adjoint $\pR{x}^*$ is an element of $\End H^+$.  (In the case that $H^+$ or $H^-$ is commutative, we often omit the superscript $L$ or $R$.)  In this way we have $\kk$-algebra homomorphisms
\begin{gather}
  H^+ \hookrightarrow \End H^+,\quad a \mapsto \pL{a}, \label{eq:H+action} \\
  H^- \hookrightarrow \End H^+,\quad x \mapsto \pR{x}^*. \label{eq:H-action}
\end{gather}
The action of $H^-$ on $H^+$ given by~\eqref{eq:H-action} is called the \emph{left regular action}.  The maps~\eqref{eq:H+action} and~\eqref{eq:H-action} are both injective.

Since $H^+ = \bigoplus_{\lambda \in \Lambda} H_\lambda^+$ is $\Lambda$-graded, we have a natural algebra $G(\Lambda)$-grading $\End H^+ = \bigoplus_{\lambda\in G(\Lambda)} \End_\lambda H^+$. Here $G(\Lambda)\cong\Z^r$ is the Grothendieck group of the monoid $\Lambda$, obtained by formally adjoining inverses. It is routine to verify that the map~\eqref{eq:H+action} sends $H^+_\lambda$ to $\End_\lambda H^+$ and the map~\eqref{eq:H-action} sends $H^-_\lambda$ to $\End_{-\lambda} H^+$ for all $\lambda \in \Lambda$.

\begin{lem}
  The left regular action of $H^-$ on $H^+$ is given by
  \[
    \pR{x}^*(a) = \sum_{(a)} q^{\gamma'(|a_1|,|a_2|)} \langle x, a_2 \rangle a_1,\quad x \in H^-,\ a \in H^+.
  \]
\end{lem}

\begin{proof}
  For $x,y \in H^-$ and $a \in H^+$, we have
  \begin{multline*}
    \langle y, \pR{x}^*(a) \rangle = \langle yx, a \rangle = q^{\gamma'(|y|,|x|)} \langle y \otimes x, \Delta(a) \rangle \\
    = \sum_{(a)} q^{\gamma'(|y|,|x|)} \langle y, a_1 \rangle \langle x, a_2 \rangle = \left\langle y, \sum_{(a)} q^{\gamma'(|a_1|,|a_2|)} \langle x, a_2 \rangle a_1 \right\rangle.
  \end{multline*}
  The result then follows from the nondegeneracy of the bilinear form.
\end{proof}

\begin{lem} \label{lem:adjoint-action-on-product}
  If $x \in H^-$ and $a,b \in H^+$, then
  \[ \ts
    \pR{x}^*(ab) = \sum_{(x)} q^{\gamma''(|x_1|,|b|-|x_2|)+\gamma''(|a|,|x_2|) + \xi'(|b|-|x_2|,|x_1|)+\xi''(|a|-|x_1|,|x_2|)} \pR{x_1}^*(a) \pR{x_2}^*(b).
  \]
\end{lem}

\begin{proof}
  For $x,y\in H^-$, $a,b\in H^+$ we have
  \begin{align*}
    \langle y, \pR{x}^*(ab) \rangle &= \langle yx, ab \rangle \\
    &= q^{\gamma''(|a|,|b|)}\langle \Delta(yx), a\otimes b \rangle \\
    &=\ts q^{\gamma''(|a|,|b|)} \left\langle \sum_{(x),(y)} q^{\xi'(|y_2|,|x_1|)+\xi''(|y_1|,|x_2|)} y_1x_1\otimes y_2x_2, a\otimes b \right\rangle \\
    &=\ts \sum_{(x),(y)} q^{\gamma''(|a|,|b|) + \xi'(|b|-|x_2|,|x_1|)+\xi''(|a|-|x_1|,|x_2|)} \langle y_1x_1, a\rangle \langle y_2x_2, b \rangle \\
    &=\ts \sum_{(x),(y)} q^{\gamma''(|a|,|b|) + \xi'(|b|-|x_2|,|x_1|)+\xi''(|a|-|x_1|,|x_2|)} \langle y_1, \pR{x_1}^*(a)\rangle \langle y_2,\pR{x_2}^*(b)\rangle \\
    &=\ts \left\langle \Delta(y), \sum_{(x)} q^{\gamma''(|a|,|b|) + \xi'(|b|-|x_2|,|x_1|)+\xi''(|a|-|x_1|,|x_2|)} \pR{x_1}^*(a)\otimes \pR{x_2}^*(b) \right\rangle \\
    &=\ts \left\langle y,  \sum_{(x)} q^{\gamma''(|a|,|b|) - \gamma''(|a|-|x_1|, |b|-|x_2|) + \xi'(|b|-|x_2|,|x_1|)+\xi''(|a|-|x_1|,|x_2|)} \pR{x_1}^*(a) \pR{x_2}^*(b) \right\rangle.
  \end{align*}
  The result then follows from the nondegeneracy of the bilinear form and the fact that
  \[
    \gamma''(|a|,|b|) - \gamma''(|a|-|x_1|,|b|-|x_2|) = \gamma''(|x_1|,|b|-|x_2|) + \gamma''(|a|,|x_2|). \qedhere
  \]
\end{proof}

\begin{defin}[Compatible dual pair] \label{def:compatible-pair}
  We say that the $(q,\gamma)$-dual pair $(H^+,H^-)$ is \emph{compatible} if there is a choice of $\chi$ such that
  \begin{equation} \label{eq:compatibility}
    \chi' = - (\gamma')^T
  \end{equation}
  (recall from Remark~\ref{rem:non-unique} that $(q,\chi)$ is not uniquely determined by $H^+$).  Whenever we refer to a compatible dual pair, we will assume that we have chosen such a $(q,\chi)$.  Note that~\eqref{eq:compatibility} is equivalent to the condition that $\gamma'' = -(\xi')^T$.
  \details{$(\xi')^T = \chi' + (\gamma')^T - \gamma'' = -\gamma''$.}
\end{defin}

\begin{rem}
  Since the twisting $(q,\chi)$ of $H^+$ is not unique (see Remark~\ref{rem:non-unique}), the issue of compatibility is rather subtle.  In fact, the authors are not aware of an example of a dual pair that is not compatible.
\end{rem}

For the remainder of this section, we assume that the dual pair $(H^+,H^-)$ is compatible and that $\chi$ satisfies~\eqref{eq:compatibility}.

\begin{cor} \label{cor:compatible-commutation}
  For all $x \in H^-$ and $a \in H^+$, we have
  \[ \ts
    \pR{x}^* \pL{a} = \sum_{(x)} q^{\gamma''(|a|,|x_2|)+\xi''(|a|-|x_1|,|x_2|)} \pL{\left(\pR{x_1}^*(a)\right)} \pR{x_2}^*.
  \]
\end{cor}

\begin{proof}
  For $a, b \in H^+$ and $x \in H^-$, we have, by Lemma~\ref{lem:adjoint-action-on-product},
  \begin{equation}\label{eq:compatible-action} \ts
    \pR{x}^*(ab) = \sum_{(x)} q^{\gamma''(|a|,|x_2|)+\xi''(|a|-|x_1|,|x_2|)} \pR{x_1}^*(a) \pR{x_2}^*(b).
  \end{equation}
  and the result follows.
\end{proof}

We see from Corollary~\ref{cor:compatible-commutation} that the compatibility of the dual pair ensures that, in the sum~\eqref{eq:compatible-action}, the coefficients are independent of $b$.  Hence we obtain a nice description of how the operators of left multiplication by $H^+$ and the left regular action of $H^-$ commute.

\begin{defin}[Twisted Heisenberg double] \label{def:h}
  We define the \emph{twisted Heisenberg double}\linebreak $\fh(H^+,H^-)$ as follows.  We set $\fh(H^+,H^-) = H^+ \otimes H^-$ as $\kk$-modules, and we write $a \# x$ for $a \otimes x$, $a \in H^+$, $x \in H^-$, viewed as an element of $\fh(H^+,H^-)$.  Multiplication is given by
  \begin{align}\label{eq:smash-product} \ts
    (a \# x)(b \# y) &:= \ts \sum_{(x)} q^{\gamma''(|b|,|x_2|) + \xi''(|b|-|x_1|,|x_2|)} a \pR{x_1}^*(b) \# x_2 y \\
    \notag &= \ts \sum_{(x),(b)} q^{\gamma''(|b|,|x_2|) + \xi''(|b|-|x_1|,|x_2|) + \gamma'(|b_1|,|b_2|)} \langle x_1, b_2 \rangle ab_1 \# x_2 y.
  \end{align}
  We will often view $H^+$ and $H^-$ as subalgebras of $\fh(H^+,H^-)$ via the maps $a \mapsto a \# 1$ and $x \mapsto 1 \# x$ for $a \in H^+$ and $x \in H^-$.  Then we have $ax = a \# x$.  When the context is clear, we will simply write $\fh$ for $\fh(H^+,H^-)$.  We have a natural grading $\fh = \bigoplus_{\lambda \in G(\Lambda)} \fh_\lambda$, where $\fh_\lambda = \bigoplus_{\mu - \nu = \lambda} H_\mu^+ \# H_\nu^-$.
\end{defin}

\begin{rem}
  \begin{enumerate}
    \item The associativity of the product~\eqref{eq:smash-product} can be shown directly.  However, since it will also follow from the fact that $\fh(H^+,H^-)$ is isomorphic to a subalgebra of $\End H^+$ (see Remark~\ref{rem:h-subalg-End-H}), we omit a direct proof, which is somewhat lengthy.

    \item If the twisted pairing is nondegenerate but not a perfect pairing (which can only happen if $\kk$ is not a field), one can still define an algebra as in Definition~\ref{def:h}.  This algebra will be a subalgebra of the twisted Heisenberg double of $H^+$, since $H^-$ will be isomorphic to a Hopf subalgebra of the Hopf algebra dual to $H^+$.
  \end{enumerate}
\end{rem}

%
\section{Fock space} \label{sec:Fock-space}
%

We now introduce a natural representation of the twisted Heisenberg double.  Throughout this section we assume that $(H^+,H^-)$ is a compatible $(q,\gamma)$-dual pair of twisted Hopf algebras, that $H^+$ is a $(q,\chi)$-Hopf algebra, and that $H^-$ is $(q,\xi)$-Hopf algebra with $\xi$ given by~\eqref{eq:xi-dual}.  We let $\fh = \fh(H^+,H^-)$.

\begin{defin}[Vacuum vector]
  An element $v$ of an $\fh$-module $V$ is called a \emph{lowest weight} (resp.\ \emph{highest weight}) \emph{vacuum vector} if $\kk v \cong \kk$, as $\kk$-modules, and $H^-_\lambda v = 0$ (resp.\ $H^+_\lambda v = 0$) for all $\lambda \ne 0$.
\end{defin}

\begin{defin}[Fock space]
  The algebra $\fh$ has a natural (left) representation on $H^+$ given by
  \[
    (a \# x)(b) = a \pR{x}^*(b),\quad a,b \in H^+,\ x \in H^-.
  \]
  We call this the \emph{lowest weight Fock space representation} of $\fh$ and denote it by $\cF = \cF(H^+,H^-)$.  Note that this representation is generated by the lowest weight vacuum vector $1 \in H^+$.
\end{defin}

Suppose $X^+$ is a $\Lambda$-graded subalgebra of $H^+$ that is invariant under the left regular action of $H^-$ on $H^+$.  Then $X^+ \# H^-$ is a subalgebra of $\fh$ acting naturally on $X^+$.  The following result (when $X^+=H^+$) is a generalization of the Stone--von Neumann Theorem to the setting of an arbitrary twisted Heisenberg double.  In the untwisted setting, it was proved in~\cite[Th.~2.11]{SY13}.

\begin{theo} \label{theo:Fock-space-properties}
  Let $X^+$ be a subalgebra of $H^+$ that is invariant under the left regular action of $H^-$ on $H^+$.
  \begin{enumerate}
    \item \label{theo-item:Fock-space-subreps} The only $(X^+ \# H^-)$-submodules of $X^+$ are those of the form $I X^+$ for some ideal $I$ of $\kk$.

    \item Let $\kk^- \cong \kk$ (isomorphism of $\kk$-modules) be the representation of $H^-$ such that $H^-_\lambda$ acts as zero for all $\lambda \neq 0$ and $H^-_0 \cong \kk$ acts by left multiplication.  Then $X^+$ is isomorphic to the induced module $\Ind^{X^+ \# H^-}_{H^-} \kk^- = (X^+ \# H^-) \otimes_{H^-} \kk^-$ as an $(X^+ \# H^-)$-module.

    \item \label{theo-item:Stone-von-Neumann} Any $(X^+ \# H^-)$-module generated by a lowest weight vacuum vector is isomorphic to $X^+$.
  \suspend{enumerate}
  If $X^+=H^+$ then $X^+ \# H^- = \fh$ and the module $X^+$ is the lowest weight Fock space $\cF$.  In that case we also have the following.
  \resume{enumerate}
    \item \label{theo-item:Fock-space-faithful} The lowest weight Fock space representation $\cF$ of $\fh$ is faithful.
  \end{enumerate}
\end{theo}

\begin{proof}
  The proof is the same as that of \cite[Th.~2.11]{SY13} except that, in the proof of part~\eqref{theo-item:Fock-space-subreps}, we use the partial order on $\Lambda$ defined by
  \[
    \lambda\leq\mu \iff \exists\ \nu \in \Lambda \text{ such that } \mu=\lambda+\nu.
   \]
  This partial order generalizes the usual order on the natural numbers used in the proof of \cite[Th.~2.11]{SY13}.
  \details{
    Details for part~\eqref{theo-item:Fock-space-subreps} are as follows.  Clearly, if $I$ is an ideal of $\kk$, then $I X^+$ is a submodule of $X^+$.  Now suppose $W \subseteq X^+$ is a nonzero submodule, and let
    \[
      I = \{c \in \kk\ |\ c1 \in W\}.
    \]
    It is easy to see that $I$ is an ideal of $\kk$.  We claim that $W = I X^+$.  Since the element $1$ generates $X^+$, we clearly have $I X^+ \subseteq W$.  Now suppose there exists $a \in W$ such that $a \not \in I X^+$.  Consider the partial order on $\Lambda$ defined above. Write $a=\sum_{\lambda\in\Lambda} a_\lambda$ and let $\mu \in \Lambda$ be maximal such that $a_\mu\neq 0$.  We may assume that $a_\mu\not \in I X^+$ (otherwise, consider $a-a_\mu$).  Let $b_1,\dotsc,b_m$ be a basis of $H_\mu^+$ such that $b_1,\dotsc,b_k$ is a basis of $X_\mu^+$, for $k = \dim_{\kk} X_\mu^+$.  Let $x_1,\dotsc,x_m$ be the dual basis of $H_\mu^-$.  Then it is easy to verify that $\sum_{j=1}^k b_j \# x_j$ acts as the identity on $X_\mu^+$ and as zero on $X_\lambda^+$ if $\lambda\not > \mu$. Thus,
    \[ \ts
      a_\mu = \sum_{j=1}^k (b_j \# x_j)(a) \in \sum_{j=1}^k b_j I X^+ \subseteq I X^+,
    \]
    since $\pR{x_j}^*(a) \in W \cap H_0^+$ for all $j=1,\dotsc,k$.  This contradiction completes the proof.
  }
\end{proof}

\begin{rem} \label{rem:h-subalg-End-H}
  By Theorem~\ref{theo:Fock-space-properties}\eqref{theo-item:Fock-space-faithful}, we may view $\fh$ as the subalgebra of $\End H^+$ generated by $\pL{a}$, $a \in H^+$, and $\pR{x}^*$, $x \in H^-$.
\end{rem}

%
\section{Shifting the product and coproduct} \label{sec:shift}
%

In categorification via towers of algebras, where products and coproducts come from induction and restriction functors between categories of graded modules, it is common to introduce a grading shift in order to obtain categorifications of specific relations.  In this section, we examine how shifting the product and coproduct of a twisted Hopf algebra changes the constructions introduced above.  In particular, we will see in Theorem~\ref{theo:alpha-invariant} that certain shifts do not change the twisted Heisenberg double $\fh(H^+,H^-)$, even though they change the Hopf algebras $H^+$ and $H^-$.  Throughout this section, we fix an invertible element $q$ of $\kk$.

For a $\Lambda$-graded $\kk$-module $M$ and $\lambda \in \Lambda$, let $P_\lambda$ denote projection onto the summand of degree $\lambda$:
\[ \ts
  P_\lambda \colon \bigoplus_{\mu \in \Lambda} M_\mu \to M_\lambda,\quad P_\lambda((m_\mu)_\mu) = m_\lambda.
\]
If $\Delta$ is a coproduct and $\alpha \colon \Lambda \times \Lambda \to \Z$ is a biadditive map, define the \emph{shifted coproduct}
\begin{equation} \ts
  \Delta_\alpha = \left( \sum_{\lambda,\mu \in \Lambda} q^{\alpha(\lambda,\mu)} P_\lambda \otimes P_\mu \right) \circ \Delta.
\end{equation}
Similarly, if $\nabla$ is a product and $\beta \colon \Lambda \times \Lambda \to \Z$ is a biadditive map, define the \emph{shifted product}
\begin{equation} \ts
  \nabla_\beta = \nabla \circ \left( \sum_{\lambda,\mu \in \Lambda} q^{\beta(\lambda,\mu)} P_\lambda \otimes P_\mu \right).
\end{equation}

\begin{prop} \label{prop:shift-Hopf-twisting}
  If $(H,\nabla,\Delta,\varepsilon,\eta)$ is a $(q,\chi)$-bialgebra and $\alpha,\beta \colon \Lambda \times \Lambda \to \Z$ are biadditive maps, then $(H,\nabla_\beta,\Delta_\alpha,\varepsilon,\eta)$ is a $(q,(\chi' + \alpha^T + \beta, \chi'' + \alpha + \beta))$-bialgebra.
\end{prop}

\begin{proof}
  Since $\nabla$ is associative, we have $\nabla(\nabla \otimes \id) = \nabla(\id \otimes \nabla)$.  Thus,
  \begin{align*}
    \nabla_\beta (\nabla_\beta \otimes \id) &= \ts \sum_{\lambda,\mu,\nu,\rho \in \Lambda} q^{\beta(\lambda,\mu) + \beta(\nu,\rho)} \nabla (P_\nu \otimes P_\rho) (\nabla \otimes \id) (P_\lambda \otimes P_\mu \otimes \id) \\
    &= \ts \sum_{\lambda,\mu,\rho \in \Lambda} q^{\beta(\lambda,\mu) + \beta(\lambda + \mu,\rho)} \nabla (\nabla \otimes \id) (P_\lambda \otimes P_\mu \otimes P_\rho) \\
    &= \ts \sum_{\lambda,\mu,\rho \in \Lambda} q^{\beta(\lambda,\mu+\rho) + \beta(\mu,\rho)} \nabla (\id \otimes \nabla) (P_\lambda \otimes P_\mu \otimes P_\rho) \\
    &= \ts \nabla_\beta (\id \otimes \nabla_\beta).
  \end{align*}
  Hence $\nabla_\beta$ is associative.  The proof that $\Delta_\alpha$ is coassociative is analogous.
  \details{
    Since $\Delta$ is coassociative, we have $(\id \otimes \Delta)\Delta = (\Delta \otimes \id)\Delta$.  Thus,
    \begin{align*}
      (\id \otimes \Delta_\alpha) \Delta_\alpha
      &= \ts \sum_{\lambda,\mu,\nu,\rho \in \Lambda} q^{\alpha(\lambda,\mu) + \alpha(\nu,\rho)} (\id \otimes P_\lambda \otimes P_\mu) (\id \otimes \Delta) (P_\nu \otimes P_\rho) \Delta \\
      &= \ts \sum_{\lambda,\mu,\nu \in \Lambda} q^{\alpha(\lambda,\mu) + \alpha(\nu,\lambda+\mu)} (P_\nu \otimes P_\lambda \otimes P_\mu) (\id \otimes \Delta) \Delta \\
      &= \ts \sum_{\lambda,\mu,\nu \in \Lambda} q^{\alpha(\lambda+\nu,\mu) + \alpha(\nu,\lambda)} (P_\nu \otimes P_\lambda \otimes P_\mu) (\Delta \otimes \id) \Delta \\
      &= (\Delta_\alpha \otimes \id) \Delta_\alpha.
    \end{align*}
    Hence $\Delta_\alpha$ is coassociative.
  }

  For the remainder of this proof, juxtaposition corresponds to the multiplication $\nabla$.  For homogeneous elements $a,b \in H$, we have
  \begin{align*}
    \Delta_\alpha& \nabla_\beta(a \otimes b) = \ts \left( \sum_{\lambda,\mu \in \Lambda} q^{\alpha(\lambda,\mu)} P_\lambda \otimes P_\mu \right) \Delta \nabla \left( \sum_{\lambda,\mu \in \Lambda} q^{\beta(\lambda,\mu)} P_\lambda \otimes P_\mu \right) (a \otimes b) \\
    &= \ts q^{\beta(|a|,|b|)} \left( \sum_{\lambda,\mu \in \Lambda} q^{\alpha(\lambda,\mu)} P_\lambda \otimes P_\mu \right) \Delta(a) *_\chi \Delta(b) \\
    &= \ts q^{\beta(|a|,|b|)} \left( \sum_{\lambda,\mu \in \Lambda} q^{\alpha(\lambda,\mu)} P_\lambda \otimes P_\mu \right) \sum_{(a),(b)} q^{\chi'(|a_2|,|b_1|) + \chi''(|a_1|,|b_2|)} a_1 b_1 \otimes a_2 b_2 \\
    &= \ts q^{\beta(|a|,|b|)} \sum_{(a),(b)} q^{\alpha(|a_1|+|b_1|,|a_2|+|b_2|) + \chi'(|a_2|,|b_1|) + \chi''(|a_1|,|b_2|)} a_1 b_1 \otimes a_2 b_2 \\
    &= \ts q^{\beta(|a|,|b|)} \sum_{(a),(b)} q^{\chi'(|a_2|,|b_1|) + \chi''(|a_1|,|b_2|) + \alpha(|a_1|,|b_2|) + \alpha(|b_1|,|a_2|)} \left( q^{\alpha(|a_1|,|a_2|)} a_1 \otimes a_2 \right) \left( q^{\alpha(|b_1|,|b_2|)} b_1 \otimes b_2 \right) \\
    &= \Delta_\alpha(a) *_{(\chi' + \alpha^T + \beta, \chi'' + \alpha + \beta)} \Delta_\alpha(b),
  \end{align*}
  where, in the last equality, we use the fact that
  \[
    q^{\beta(|a|,|b|)} = q^{\beta(|a_1|,|a_2|)} q^{\beta(|b_1|,|b_2|)} q^{\beta(|a_1|,|b_2|)} q^{\beta(|a_2|,|b_1|)},
  \]
  and the factors $q^{\beta(|a_1|,|b_1|)}$ and $q^{\beta(|a_2|,|b_2|)}$ are absorbed into the products, under $\nabla_\beta$, of $a_1$, $b_1$ and $a_2$, $b_2$.

  Using the fact that any biadditive map $\Lambda \times \Lambda \to \Z$ takes the value zero when either of the arguments is equal to zero, it is straightforward to verify that the remaining axioms of a twisted bialgebra, which involve the unit and counit, are satisfied.
\end{proof}

\begin{lem} \label{lem:shift-pairing-twisting}
  Suppose $(H^+,H^-)$ is a $(q,\gamma)$-dual pair of twisted Hopf algebras.  Let $\tilde H^\pm$ be obtained from $H^\pm$ by replacing the coproduct of $H^\pm$ by $\Delta_{\alpha^\pm}$ for biadditive maps $\alpha^\pm \colon \Lambda \times \Lambda \to \Z$ and the product by $\nabla_{\beta^\pm}$ for biadditive maps $\beta^\pm \colon \Lambda \times \Lambda \to \Z$.  Then $(\tilde H^+, \tilde H^-)$ is a $(q,(\gamma' -\alpha^+ +\beta^-,\gamma'' -\alpha^- + \beta^+))$-dual pair of twisted Hopf algebras.
\end{lem}

\begin{proof}
  For homogeneous $x,y \in H^-$ and $a,b \in H^+$, we have
  \begin{multline*}
    \langle \nabla_{\beta^-} (x \otimes y), a \rangle = q^{\beta^-(|x|,|y|)} \langle \nabla(x \otimes y), a \rangle = q^{\gamma'(|x|,|y|) + \beta^-(|x|,|y|)} \langle x \otimes y, \Delta(a) \rangle \\
    = q^{\gamma'(|x|,|y|) - \alpha^+(|x|,|y|) + \beta^-(|x|,|y|)} \langle x \otimes y, \Delta_{\alpha^+}(a) \rangle
  \end{multline*}
  and
  \begin{multline*}
    \langle x, \nabla_{\beta^+}(a \otimes b) \rangle = q^{\beta^+(|a|,|b|)} \langle x, \nabla(a \otimes b) \rangle \\
    = q^{\gamma''(|a|,|b|) + \beta^+(|a|,|b|)} \langle \Delta(x), a \otimes b \rangle = d^{\gamma''(|a|,|b|) - \alpha^-(|a|,|b|) + \beta^+(|a|,|b|)} \langle \Delta_{\alpha^-}(x), a \otimes b \rangle. \qedhere
  \end{multline*}
\end{proof}

\begin{lem}
  Suppose $(H^+,H^-)$ is a compatible dual pair of twisted Hopf algebras and define $\tilde H^\pm$ as in Lemma~\ref{lem:shift-pairing-twisting}.  Then $(\tilde H^+, \tilde H^-)$ is a compatible dual pair if $\beta^+ = -(\beta^-)^T$.
\end{lem}

\begin{proof}
  By Proposition~\ref{prop:shift-Hopf-twisting} and Lemma~\ref{lem:shift-pairing-twisting}, we have
  \[
    \tilde \chi' = \chi' + (\alpha^+)^T + \beta^+,\quad \tilde \gamma' = \gamma' - \alpha^+ + \beta^-.
  \]
  Since $(H^+,H^-)$ is a compatible dual pair, we have $\chi' = - (\gamma')^T$.  Thus
  \[
    \tilde \chi' - (\alpha^+)^T - \beta^+ = \chi' = -(\gamma')^T = -(\tilde \gamma')^T - (\alpha^+)^T + (\beta^-)^T
  \]
  and so $\tilde \chi' = - (\tilde \gamma')^T$ if $\beta^+ = -(\beta^-)^T$.
\end{proof}

\begin{theo} \label{theo:alpha-invariant}
  Suppose $(H^+,H^-)$ is a compatible dual pair of twisted Hopf algebras and define $\tilde H^\pm$ as in Lemma~\ref{lem:shift-pairing-twisting}, with $\alpha^+ = \alpha^-$ and $\beta^+=\beta^-=0$.  Then  $\fh(H^+,H^-) \cong \fh(\tilde H^+, \tilde H^-)$ as algebras.
\end{theo}

\begin{proof}
  Let $\alpha = \alpha^+ = \alpha^-$.  We have
  \[
    \tilde \gamma'' = \gamma'' - \alpha \quad \text{and} \quad \tilde \xi'' = \tilde \chi'' + \tilde \gamma' - \tilde \gamma'' = \chi'' + \alpha + \gamma' - \alpha - \gamma'' + \alpha = \xi'' + \alpha.
  \]
  Now, for homogeneous $x \in H^-$, if $\Delta(x) = \sum_{(x)} x_1 \otimes x_2$ for homogeneous $x_1,x_2$, then $\Delta_\alpha(x) = \sum_{(x)} q^{\alpha(|x_1|,|x_2|)} x_1 \otimes x_2$.  Thus, the multiplication in $\fh(\tilde H^+, \tilde H^-)$ is given by
  \begin{align*}
    (a \# x)(b \# y) &= \ts \sum_{(x)} q^{\gamma''(|b|,|x_2|) -\alpha(|b|,|x_2|) + \xi''(|b|-|x_1|, |x_2|) + \alpha(|b|-|x_1|,|x_2|) + \alpha(|x_1|, |x_2|)} a \pR{x_1}^*(b) \# x_2y \\
    &= \ts \sum_{(x)} q^{\gamma''(|b|,|x_2|) + \xi''(|b|-|x_1|, |x_2|)} a \pR{x_1}^*(b) \# x_2y,
  \end{align*}
  which is the multiplication in $\fh(H^+,H^-)$.
\end{proof}

For convenience, we summarize here the relations found above, with the notation of Lemma~\ref{lem:shift-pairing-twisting}, and letting $(q,\chi)$ and $(q,\xi)$ be the twistings of $H^+$ and $H^-$, respectively, chosen to satisfy~\eqref{eq:xi-dual}:
\begin{gather}
  \xi' = (\chi')^T + \gamma' - (\gamma'')^T,\quad \xi'' = \chi'' + \gamma' - \gamma'', \\
  \tilde \chi' = \chi' + (\alpha^+)^T + \beta^+,\quad \tilde \chi'' = \chi'' + \alpha^+ + \beta^+, \\
  \tilde \xi' = \xi' + (\alpha^-)^T + \beta^-,\quad \tilde \xi'' = \xi'' + \alpha^- + \beta^-, \\
  \tilde \gamma' = \gamma' - \alpha^+ + \beta^-,\quad \tilde \gamma'' = \gamma'' - \alpha^- + \beta^+.
\end{gather}

%
\section{The quantum Weyl algebra} \label{sec:quantum-Weyl}
%

Let $\kk=\Q(q)$, where $q$ is an indeterminate.  For a positive integer $n$ and nonnegative integer $k$, $0 \le k \le n$, define
\[ \ts
  [n]_q = 1 + q + \dotsb + q^{n-1},\quad [n]_q! = \prod_{i=1}^n [i]_q,\quad \qbin{n}{k}_q = \frac{[n]_q!}{[k]_q! [n-k]_q!},
\]
where, by convention, we set $[0]_q=0$ and $[0]_q!=1$.  We have
\[ \ts
  [n]_{q^{-1}} = q^{-(n-1)}[n]_q,\quad [n]_{q^{-1}}! = q^{-\binom{n}{2}} [n]_q!,\quad \qbin{n}{k}_{q^{-1}} = q^{-k(n-k)} \qbin{n}{k}_q.
\]
Consider $\kk[x]$, which is $\N$-graded by degree. Define
\begin{gather*}
  \nabla \colon \kk[x] \otimes \kk[x] \to \kk[x],\quad \nabla(x^m \otimes x^n) = x^{m+n}, \\ \ts
  \Delta \colon \kk[x] \to \kk[x] \otimes \kk[x],\quad \Delta (x^n) = \sum_{k=0}^n \qbin{n}{k}_q x^k \otimes x^{n-k},
\end{gather*}
extended by linearity.  We let $\varepsilon \colon \kk[x] \to \kk$ and $\eta \colon \kk \to \kk[x]$ be the natural projection and inclusion maps.  Define $\zeta \colon \N \times \N \to \Z$ by $\zeta(m,n) = mn$.  We have
\begin{align*}
  \Delta(x^m) *_{(0,\zeta)} \Delta(x^n)
  &= \ts \left( \sum_{\ell=0}^m \qbin{m}{\ell}_q x^\ell \otimes x^{m-\ell} \right) *_{(0,\zeta)} \left( \sum_{k=0}^n \qbin{n}{k}_q x^k \otimes x^{n-k} \right) \\
  &= \ts \sum_{\ell=0}^m \sum_{k=0}^n \qbin{m}{\ell}_q \qbin{n}{k}_q q^{\ell(n-k)} x^{\ell+k} \otimes x^{m+n-\ell-k} \\
  &= \ts \sum_{j=0}^{m+n} \left(\sum_{\ell + k = j} \qbin{m}{\ell}_q \qbin{n}{k}_q q^{\ell(n-k)} \right) x^j \otimes x^{m+n-j} \\
  &= \ts \sum_{j=0}^{m+n} \qbin{m+n}{j} x^j \otimes x^{m+n-j} \\
  &= \Delta (x^{m+n}),
\end{align*}
where, in the second-to-last equality, we used the $q$-Vandermonde identity.  Thus, $(\kk[x],\nabla,\Delta,\varepsilon,\eta)$ is a $(q,0,\zeta)$-Hopf algebra.  It can similarly be seen to be a $(q,\zeta,0)$-Hopf algebra.

Let $H^+ = \kk[x]$, with twisting $(q,\chi)$, $\chi = (\chi',\chi'') = (0,\zeta)$.  We also let $H^- = \kk[\partial]$, defined as above with $x$ replaced by $\partial$ and $q$ by $q^{-1}$, and with twisting $(q,\xi)$, $\xi = (\xi',\xi'') = (-\zeta, 0)$.  Define the $\kk$-bilinear form
\[
  \langle -, - \rangle \colon H^- \otimes H^+ \to \kk,\quad \langle \partial^m, x^n \rangle = \delta_{mn} [n]_q!.
\]
Then
\begin{align*}
  \langle \partial^m \otimes \partial^n, \Delta(x^{m+n}) \rangle
  &= \ts \left\langle \partial^m \otimes \partial^n, \sum_{k=0}^{m+n} \qbin{m+n}{k}_q x^k \otimes x^{m+n-k} \right\rangle \\
  &= \ts \qbin{m+n}{m}_q \langle \partial^m \otimes \partial^n, x^m \otimes x^n \rangle \\
  &= [m+n]_q! \\
  &= \langle \partial^m \partial^n, x^{m+n} \rangle
\end{align*}
and
\begin{align*}
  \langle \Delta(\partial^{m+n}), x^m \otimes x^n \rangle\
  &= \ts \left\langle \sum_{k=0}^{m+n} \qbin{m+n}{k}_{q^{-1}} \partial^k \otimes \partial^{m+n-k}, x^m \otimes x^n \right\rangle \\
  &= \ts \qbin{m+n}{m}_{q^{-1}} \langle \partial^m \otimes \partial^n, x^m \otimes x^n \rangle \\
  &= \ts q^{-mn}\qbin{m+n}{m}_q \langle \partial^m \otimes \partial^n, x^m \otimes x^n \rangle \\
  &= \ts q^{-mn} [m+n]_q! \\
  &= q^{-mn} \langle \partial^{m+n}, x^m x^n \rangle.
\end{align*}
Thus $\langle -, - \rangle$ is a $(q,\gamma)$-twisted pairing with $\gamma = (0,\zeta)$.  Since $\gamma'=0=\chi'$, the dual pair is compatible.  Because $\xi''=0$, in the twisted Heisenberg double we have
\[
  \partial x = (1 \# \partial)(x \# 1) = q^{|x||\partial|} 1^*(x) \# \partial + q^{|x||1|} \partial^*(x) \# 1 = q x \# \partial + q^0 1 \# 1 = q x \partial + 1.
\]
Thus we see that
\[
  \fh(H^+,H^-) = \kk \langle x, \partial\ |\ \partial x = q x \partial + 1 \rangle
\]
is the \emph{quantum Weyl algebra}.  Its Fock space representation is the representation on $\kk[x]$ given by
\[
  x \cdot x^n = x^{n+1},\quad \partial \cdot x^n = [n]_q x^{n-1},\quad n \in \N.
\]

%
\section{Quantum Heisenberg algebras} \label{sec:q-Heis}
%

In this section, we show that quantum Heisenberg algebras can be realized as Heisenberg doubles.  In fact, in this case, the twistings are trivial.  Nevertheless, we include a discussion of quantum Heisenberg algebras in the current paper for two reasons.  First, the categorification of these algebras appearing in the literature (see, for example,~\cite{CL12}) involves categories of \emph{graded} modules, which form the motivation for our introduction of twisted Heisenberg doubles (since, in general, such categorifications involve nontrivial twistings).  Second, the twisted Heisenberg double point of view allows us to determine the effect of shifting various induction and restriction functors, using the results of Section~\ref{sec:shift}.

Fix a finite set $I$ and a symmetric map $\langle -, - \rangle \colon I \times I \to \Z$.  Then we have a symmetric matrix $A$ whose $(i,j)$ entry is $A_{ij} = \langle i, j \rangle$ for $i,j \in I$.  For $n \in \N$, define the symmetric quantum integer
\[
  [n] = \frac{q^{-n} - q^n}{q^{-1} - q} = q^{-n+1} + q^{-n+3} + \dotsb + q^{n-3} + q^{n-1},
\]
where $q$ is an indeterminate.  We also define $[-n] = (-1)^{n+1}[n]$ for $n \in \N_+$.  Note that this is different than the quantum integer $[n]_q$ defined in Section~\ref{sec:quantum-Weyl}.  One could modify the construction in Section~\ref{sec:quantum-Weyl} by a shift (see Section~\ref{sec:shift}) in order to use the symmetric quantum integer $[n]$, but we use different choices in the two settings to obtain presentations that are more natural from the point of view of categorification.

\begin{lem}
  Suppose one of the following conditions hold:
  \begin{enumerate}
    \item The matrix $A$ is nonsingular (e.g.\ $A$ is a Cartan matrix of finite ADE type).
    \item The matrix $A$ is a Cartan matrix of affine ADE type other than type $A_2^{(1)}$.
  \end{enumerate}
  Then the matrix $([k \langle i, j \rangle])_{i,j \in I}$ is nonsingular for all $k \in \N_+$.
\end{lem}

\begin{proof}
  Let $B$ be the matrix whose $(i,j)$ entry is $[k\langle i,j \rangle]$.  Thus, we wish to show that $B z \ne 0$ for all nonzero $z \in \Z[q,q^{-1}]^I$ (hence also for all $z \in \Q(q)^I$).  Suppose, on the contrary, that $Bz = 0$ for some nonzero $z \in \Z[q,q^{-1}]^I$.  Taking $q=1$ would then yield a null vector for $kA$.  Thus, if $A$ is nonsingular, so is $B$.  It remains to consider affine ADE type (other that type $A_2^{(1)}$) and $z$ equal to the unique (up to scalar multiple) null vector of $A$.  Direct computation then shows that, in each case, $Bz \ne 0$.  For instance, in type A, we can take $I = \{0,\dotsc,n\}$ and $B$ to be the matrix whose $(i,j)$ entry is
  \[
    B_{ij} =
    \begin{cases}
      [2k] & i = j, \\
      [-k] & i-j \equiv \pm 1 \mod n+1, \\
      0 & \text{otherwise},
    \end{cases}
  \]
  and $z = (1,\dotsc,1)$.  Then every entry of $Bz$ is $[2k] + 2[-k] \ne 0$.  Types D and E are similar.
\end{proof}

For the remainder of this section, unless otherwise indicated, we assume that the matrix $([k \langle i, j \rangle])_{i,j \in I}$ is nonsingular for all $k \in \N_+$.

For convenience of notation, we identify $I$ with the set $\{1,2,\dotsc,|I|\}$. Let $\kk = \Q(q)$, and define the graded Hopf algebra
\[
  H^+ = \Sy^{\otimes I} = \bigotimes_{i \in I} \Sy,
\]
where $\Sy$ is the algebra of symmetric functions over $\kk$ and the usual grading on $\Sy$ (by degree) induces the grading on $H^+$.  For $n \in \N$ and $i \in I$, let $p_{n,i}$ denote the $n$-th power sum in the $i$-th factor of $H^+$.  Then we have an isomorphism of $\kk$-algebras
\[
  H^+ \cong \kk[p_{n,i}\ |\ n \in \N,\ i \in I],
\]
and
\begin{equation} \label{eq:p-coproduct}
  \Delta(p_{n,i}) = p_{n,i} \otimes 1 + 1 \otimes p_{n,i},\quad n \in \N_+,\ i \in I.
\end{equation}
We adopt the convention that $p_{0,i}=1$ and $p_{n,i}=0$ for $n<0$ and $i \in I$.  Let $\cP$ denote the set of partitions.  For a partition $\lambda \in \cP$, and $k \in \N$, we let $m_k(\lambda)$ denote the number of parts of $\lambda$ equal to $k$.  For $\lambda = (\lambda_1,\lambda_2,\dotsc) \in \cP$ and $i \in I$, define
\[
  p_{\lambda,i} = \prod_{k=1}^{\ell(\lambda)} p_{\lambda_k,i},
\]
where $\ell(\lambda)$ denotes the length of $\lambda$ (i.e.\ the number of nonzero parts).

We will use bold Greek symbols $\blambda$, $\bmu$, etc.\ to denote elements of $\cP^I$, and the $i$-th component of such an element will be denoted by a superscript $i$ on the corresponding non-bold letter.  For example, $\lambda^i$ is the $i$-th component of $\blambda$, and the $k$-th part of $\lambda^i$ is $\lambda^i_k$.  For $\blambda \in \cP^I$, let
\[
  p_\blambda = \prod_{i \in I} p_{\lambda^i,i} = \prod_{i \in I} \prod_{k=1}^{\ell(\lambda^i)} p_{\lambda^i_k,i}.
\]
Then the $p_\blambda$, $\blambda \in \cP^I$, form a basis of $\Sy^{\otimes I}$.

For $\blambda \in \cP^I$, define $\underline{\blambda}$ to be the sequence of elements of $\N_+ \times I$, in lexicographical order, such that the element $(k,i)$ appears $m_k(\lambda^i)$ times.  In other words $\underline{\blambda}$ consists of all the parts of the $\lambda^i$, $i \in I$, with their color $i$ recorded.  For example, if $|I| = 3$ and $\blambda = (\lambda^1,\lambda^2,\lambda^3) = ((3,2,1),(2,2),(5,4,3,1))$, then
\[
  \underline{\blambda} = ((1,1),(1,3),(2,1),(2,2),(2,2),(3,1),(3,3),(4,3),(5,3)).
\]
We let $\ell(\underline{\blambda}) = \sum_{i \in I} \ell(\lambda^i)$ be the length of the sequence $\underline{\blambda}$.  We will write the $n$-th term of the sequence $\underline{\blambda}$ as
\[
  \underline{\blambda}_n = (\prt(\underline{\blambda}_n),\clr(\underline{\blambda}_n)) \in \N_+ \times I,\quad \text{for } n = 1,\dotsc,\ell(\underline{\blambda}).
\]
We define a symmetric bilinear form on $H^+$ by
\begin{equation}  \label{eq:q-Heis-bilinear-form}
  \langle p_\blambda, p_\bmu \rangle = \delta_{\ell(\underline{\blambda}), \ell(\underline{\bmu})} \sum_{\sigma \in S_{\ell(\underline{\blambda})}} \prod_{r=1}^{\ell(\underline{\blambda})} \delta_{\prt \left( \underline{\blambda}_r \right), \prt \left( \underline{\bmu}_{\sigma(r)} \right)} \left[ \prt \left( \underline{\blambda}_r \right) \left\langle \clr\left(\underline{\blambda}_r\right), \clr\left(\underline{\bmu}_{\sigma(r)}\right) \right\rangle \right] \frac{\left[\prt\left(\underline{\blambda}_r\right)\right]}{\prt\left(\underline{\blambda}_r\right)}.
\end{equation}
Note that this implies that $\langle p_\blambda, p_\bmu \rangle = 0$ unless, for each $k \in \N_+$, we have $\sum_{i \in I} m_k(\lambda^i) = \sum_{i \in I} m_k(\mu^i)$.  In particular, we have
\[
  \langle p_{\lambda,i}, p_{\mu,j} \rangle = \delta_{\lambda \mu} \prod_{k \ge 1} \left( [k \langle i,j \rangle] \frac{[k]}{k} \right)^{m_k(\lambda)} m_k(\lambda)!.
\]

For a partition $\lambda = (\lambda_1,\dotsc,\lambda_\ell) \in \cP$ with $\lambda_i = k$ for some $k \in \N$, define $\lambda \ominus k = (\lambda_1,\dotsc,\lambda_{i-1},\lambda_{i+1},\dotsc,\lambda_\ell)$.  Similarly, for any partition $\lambda \in \cP$ and $k \in \N$, we define $\lambda \oplus k$ to be the partition obtained from $\lambda$ by adding a part $k$.  For another $\mu = (\mu_1,\dotsc,\mu_{\ell'}) \in \cP$, we define $\lambda \oplus \mu = (\cdots((\lambda \oplus \mu_1) \oplus \mu_2) \oplus \dotsb \oplus \mu_{\ell'})$.  If $\mu$ is a subpartition of $\lambda$, we similarly define $\lambda \ominus \mu = (\cdots((\lambda \ominus \mu_1) \ominus \mu_2) \ominus \dotsb \ominus \mu_{\ell'})$.  For $\blambda \in \cP^I$, we let $\blambda \oplus_i k$ denote the element of $\cP^I$ whose $j$-th component is $\lambda^j$ for $j \ne i$, and whose $i$-th component is $\lambda^i \oplus k$.  We define $\blambda \oplus \bmu$ to be the element of $\cP^I$ whose $j$-th component is $\lambda^j \oplus \mu^j$.  Similarly, if $\mu^j$ is a subpartition of $\lambda^j$ for all $j \in I$, then we define $\blambda \ominus \bmu$ to be the element of $\cP^I$ whose $j$-th component is $\lambda^j \ominus \mu^j$.

\begin{prop}
  The symmetric bilinear form~\eqref{eq:q-Heis-bilinear-form} is a Hopf pairing (i.e.\  a $(1,0,0)$-twisted pairing of Hopf algebras).
\end{prop}

\begin{proof}
  For $\blambda,\bnu,\bmu \in \cP^I$, we have
  \[
    \langle p_\blambda \otimes p_\bnu, \Delta(p_\bmu) \rangle = \sum_{\bmu'} \langle p_\blambda, p_{\bmu'} \rangle \langle p_\bnu, p_{\bmu \ominus \bmu'} \rangle = \langle p_{\blambda \oplus \bnu}, p_\bmu \rangle = \langle \nabla(p_\blambda \otimes p_\bnu), p_\mu \rangle,
  \]
  where the sum is over all $\bmu' \in \cP^I$ corresponding to subsequences of $\underline{\bmu}$.  The remaining axioms of a Hopf pairing are easily seen to be satisfied.
\end{proof}

For the purposes of proving some of the results below, we temporarily introduce maps $\varphi_{k,i}$, $k \in \N$, $i \in I$, as follows.  (We will see in Proposition~\ref{prop:q-Heis-form-nondegen} that $\varphi_{k,i}$ is the adjoint to multiplication by $p_{k,i}$.)  We define $\varphi_{k,i}$ to be the derivation on $H^+$ given on the generators $p_{n,j}$, $n \in \N$, $j \in I$, by $\varphi_{k,i}(p_{n,j}) = \delta_{kn} [k \langle i,j \rangle] \frac{[k]}{k}$.  Thus, for $\lambda \in \cP$ and $j \in I$, we have
\begin{equation} \label{eq:varphi-def}
  \varphi_{k,i}(p_{\lambda,j}) = m_k(\lambda) [k\langle i,j \rangle] \frac{[k]}{k} p_{\lambda \ominus k, j},
\end{equation}
where we interpret $p_{\lambda \ominus k, j}=0$ if $\lambda$ has no part equal to $k$.

\begin{lem} \label{lem:pp-adjoint-action}
  For $x \in H^+$, $k \in \N$, $\lambda \in \cP$, and $i,j \in I$, we have
  \[
    \langle p_{k,i} x, p_{\lambda,j} \rangle = \left\langle x, \varphi_{k,i}(p_{\lambda,j}) \right\rangle,
  \]
\end{lem}

\begin{proof}
  We prove the result by induction on the length of $\lambda$.  First note that, for $c \in \kk$, we have $\varphi_{k,i}(c)=0$ and $\langle p_{k,i} x, c \rangle=0$.  Thus the result holds for $\lambda$ of length zero.  Now, for $n \in \N$ and $i,j \in I$, we have $\langle p_{k,i}, p_{n,j} \rangle = \delta_{kn} [k\langle i,j \rangle] \frac{[k]}{k}$  and $\langle p_{k,i}x, p_{n,j} \rangle = 0 = \langle x, \delta_{kn} [k\langle i,j \rangle] \frac{[k]}{k} \rangle$ for $x \in H_n^+$ with $n > 0$.  Thus, the result holds for $\lambda$ of length one.

  Now assume $\lambda$ has length at least two.  Then we can choose nonempty partitions $\mu,\nu \in \cP$ with $\mu \oplus \nu = \lambda$.  By~\eqref{eq:p-coproduct}, we have, for all $x \in H^+$,
  \begin{multline*}
    \langle p_{k,i}x, p_{\lambda,j} \rangle = \langle p_{k,i} x, p_{\mu,j}p_{\nu,j} \rangle = \langle \Delta(p_{k,i} x), p_{\mu,j} \otimes p_{\nu,j} \rangle = \langle \Delta(p_{k_i}) \Delta(x), p_{\mu,j} \otimes p_{\nu,j} \rangle \\
    = \langle (p_{k,i} \otimes 1 + 1 \otimes p_{k,i}) \Delta(x), p_{\mu,j} \otimes p_{\nu,j} \rangle
    = \langle \Delta(x), \varphi_{k,i}(p_{\mu,j}) \otimes p_{\nu,j} + p_{\mu,j} \otimes \varphi_{k,i}(p_{\nu,j}) \rangle \\
    = \langle x, \varphi_{k,i}(p_{\mu,j})p_{\nu,j} + p_{\mu,j} \varphi_{k,i}(p_{\nu,j}) \rangle = \langle x, \varphi_{k,i}(p_{\mu,j}p_{\nu,j}) \rangle = \langle x, \varphi_{k,i}(p_{\lambda,j}) \rangle,
  \end{multline*}
  where we have used the inductive hypothesis in the fifth equality.  The result thus follows by induction.
\end{proof}

\begin{prop} \label{prop:q-Heis-form-nondegen}
  The bilinear form $\langle -, - \rangle$ is nondegerate and
  \begin{equation} \label{eq:pp-adjoint-action}
    p_{k,i}^*(p_{\lambda,j}) = \varphi_{k,i}(p_{\lambda,j}) = m_k(\lambda) [k\langle i,j \rangle] \frac{[k]}{k} p_{\lambda \ominus k, j},
  \end{equation}
  for all $k \in \N$, $i,j \in I$, $\lambda \in \cP$.
\end{prop}

\begin{proof}
  Recall that we have a natural grading $H^+ = \bigoplus_{n  \in \N} H_n^+$.  We prove that the form is nondegenerate by induction on the grading.  It is clearly nondegenerate on $H_0^+ = \kk$.  Now, assume that it is nondegenerate on $\bigoplus_{n = 0}^N H_n^+$ for some $N \in \N$.  Suppose for a moment that, for any nonzero $a \in H^+_{N+1}$, there exists $k \in \N_+$ and $i \in I$ such that $\varphi_{k,i}(a) \ne 0$.  Then we have $\varphi_{k,i}(a) \in \bigoplus_{n = 0}^N H_n^+$ and so, by the induction hypothesis, there exists $x \in H^+$ such that
  \[
    0 \ne \langle x, \varphi_{k,i}(a) \rangle = \langle p_{k,i}x, a \rangle,
  \]
  completing the inductive step.

  It remains to prove that, for any nonzero $a \in H_{N+1}^+$, $N \in \N$, there exists $k \in \N_+$ and $i \in I$ such that $\varphi_{k,i}(a) \ne 0$.  Suppose, on the contrary, that there exists a nonzero $a \in H_{N+1}^+$ such that $\varphi_{k,i}(a)=0$ for all $k \in \N_+$ and $i \in I$.  Write
  \[
    a = \sum_{\blambda \in \cP^I} a_\blambda p_\blambda.
  \]
  Fix $\bmu \in \cP^I$.  Considering the $p_\bmu$ coefficient of $\varphi_{k,i}(a)$, we must have, for $k \in \N_+$ and $i \in I$,
  \[
    0 = \sum_{j \in I} a_{\bmu \oplus_j k} (m_k(\mu^j)+1) [k \langle i,j \rangle] \frac{[k]}{k}.
  \]
  It then follows from the fact that the matrix $([k \langle i,j \rangle])_{i,j \in I}$ is nonsingular,
  that $a_{\bmu \oplus_j k} = 0$ for all $\bmu \in \cP^I$, $k \in \N_+$, and $j \in I$.  Since $a$ is homogeneous of positive degree, all $\blambda$ such that $a_\blambda$ is nonzero are of the form $\bmu \oplus_j k$ for some $\bmu \in \cP^I$, $k \in \N_+$ and $j \in I$.  Thus $a=0$, completing the proof by contradiction.  Equation~\eqref{eq:pp-adjoint-action} now follows immediately from Lemma~\ref{lem:pp-adjoint-action}.
\end{proof}

Since the bilinear form is nondegenerate and $\kk$ is a field, it is also a perfect pairing.  Thus, we can consider the Heisenberg double $\fh = \fh(H^+,H^-)$, with $H^-=H^+$.  For $n \in \N_+$ and $i \in I$, we let $p_{n,i}$ denote, as usual, the $n$-th power sum in the $i$-th factor of $H^+$, and let $p_{n,i}'$ denote the same element of $\Sy^{\otimes I}$, but considered as an element of $H^-$.

\begin{prop} \label{prop:q-Heis-p-presentation}
  The Heisenberg double $\fh$ is generated by $p_{n,i}$, $p_{n,i}'$, $n \in \N_+$, $i \in I$, with relations
  \begin{equation} \label{eq:q-Heis-p-presentation}
      p_{m,i} p_{n,j} = p_{n,j} p_{m,i},\quad p_{m,i}' p_{n,j}' = p_{n,j}' p_{m,i}',\quad
      p_{m,i}' p_{n,j} = p_{n,j} p_{m,i}' + \delta_{m,n} [n \langle i,j \rangle] \frac{[n]}{n}.
  \end{equation}
\end{prop}

\begin{proof}
  The $p_{n,i}$, $p_{n,i}'$ clearly generate $\fh$.  A complete set of relations for any twisted Heisenberg double is given by the relations in $H^+$, the relations in $H^-$, and the commutation relations between the elements of a generating set for $H^+$ and a generating set for $H^-$.  It thus suffices to compute the commutation relation between $p_{m,i}'$ and $p_{n,j}$ for $m,n \in \N$ and $i,j \in I$.  We have, by~\eqref{eq:smash-product}, \eqref{eq:p-coproduct}, and Proposition~\ref{prop:q-Heis-form-nondegen},
  \[
    p_{m,i}' p_{n,j} = p_{n,j} p_{m,i}' + p_{m,i}^*(p_{n,j}) = p_{n,j} p_{m,i}' + \delta_{mn} [n \langle i,j \rangle] \frac{[n]}{n}. \qedhere
  \]
\end{proof}

In the case that $A$ is a Cartan matrix of affine ADE type, the algebra with generators $p_{n,i}$, $p_{n,i}'$, $n \in \N_+$, $i \in I$, and relations given by~\eqref{eq:q-Heis-p-presentation} is sometimes called the \emph{quantum toroidal Heisenberg algebra}.  Thus, Proposition~\ref{prop:q-Heis-p-presentation} implies that the quantum toroidal Heisenberg algebra is the Heisenberg double of $\Sy^{\otimes I}$.  The Heisenberg double $\fh(H^+,H^-)$ depends, up to isomorphism, only on the Hopf algebra $H^+$.  The purpose of the bilinear form between $H^+$ and $H^-$ (e.g.\ the form~\eqref{eq:q-Heis-bilinear-form}) is simply to give the identification of $H^-$ with the Hopf algebra dual to $H^+$.  Different bilinear forms would lead to different presentations of the Heisenberg double, but the algebra itself is unchanged (up to isomorphism).  This observation leads to the following result.

\begin{cor}
  The quantum toroidal Heisenberg algebra is isomorphic to the usual infinite-dimensional Heisenberg algebra over $\Q(q)$.
\end{cor}

\begin{proof}
  The usual infinite-dimensional Heisenberg algebra is the Heisenberg double of the ring $\Sy$ of symmetric functions (see, for example, \cite[\S6.2]{SY13}).  The result then follows immediately from Proposition~\ref{prop:q-Heis-p-presentation} and the fact that $\Sy^{\otimes I}$ is isomorphic to $\Sy$ as a Hopf algebra.
\end{proof}

In the remainder of this section we will find an integral form of the Heisenberg double $\fh$.  Such integral forms are especially important in categorification.

For a partition $\lambda$, define
\[
  Z_\lambda = \prod_{k \ge 1} [k]^{m_k(\lambda)} m_k(\lambda)!,
\]
where we remind the reader that $m_k(\lambda)$ is the number of parts of $\lambda$ equal to $k$.  Note that $Z_\lambda$ is a quantized analogue of the quantity $z_\lambda = \prod_{k \ge 1} k^{m_k(\lambda)} m_k(\lambda)!$ that plays an important role in the usual theory of symmetric functions (see~\cite[\S I.2]{Mac95}).  Then, for $n \in \N$ and $i \in I$, let
\begin{equation} \label{eq:h-def}
  h_{n,i} = \sum_{\lambda \vdash n} \frac{p_{\lambda,i}}{Z_\lambda},
\end{equation}
where $\lambda \vdash n$ indicates that $\lambda$ is a partition of $n$.  Then the $h_{n,i}$ are quantized analogues of the complete symmetric functions.  In particular, replacing $Z_\lambda$ by $z_\lambda$ in~\eqref{eq:h-def} gives precisely the complete symmetric functions.  We adopt the convention that $h_{n,i} = 0$ for $n < 0$ and $i \in I$.

\begin{lem}
  For $n \in \N_+$ and $i \in I$,
  \begin{equation} \label{eq:h-hp-expansion}
    nh_{n,i} = \sum_{r=1}^n \frac{r}{[r]} h_{n-r,i} p_{r,i}.
  \end{equation}
\end{lem}

\begin{proof}
  We have
  \[
    \sum_{r=1}^n \frac{r}{[r]} h_{n-r,i} p_{r,i} = \sum_{r=1}^n \sum_{\lambda \vdash (n-r)} \frac{r}{[r]Z_\lambda} p_{\lambda, i} p_{r,i}
    = \sum_{r=1}^n \sum_{\lambda \vdash (n-r)} r m_r(\lambda \oplus r) \frac{p_{\lambda \oplus r, i}}{Z_{\lambda \oplus r}}  = n h_{n,i}. \qedhere
  \]
\end{proof}

\begin{lem} \label{lem:h-coproduct}
  For all $n \in \N$ and $i \in I$,
  \[
    \Delta(h_{n,i}) = \sum_{k=0}^n h_{k,i} \otimes h_{n-k,i},\quad n \in \N,\ i \in I.
  \]
\end{lem}

\begin{proof}
  We have
  \begin{multline*}
      \Delta(h_{n,i}) = \sum_{\lambda \vdash n} \frac{1}{Z_\lambda} \Delta(p_{\lambda,i})
      = \sum_{\lambda \vdash n} \frac{1}{Z_\lambda} \prod_{k=1}^{\ell(\lambda)} (p_{\lambda_k,i} \otimes 1 + 1 \otimes p_{\lambda_k,i}) \\
      = \sum_{\lambda \vdash n} \sum_{\lambda' \oplus \lambda'' = \lambda} \frac{1}{Z_\lambda} \prod_{k \ge 1} \frac{m_k(\lambda)!}{m_k(\lambda')! m_k(\lambda'')!} p_{\lambda',i} \otimes p_{\lambda'',i}
      = \sum_{\lambda \vdash n} \sum_{\lambda' \oplus \lambda'' = \lambda} \frac{1}{Z_{\lambda'} Z_{\lambda''}} p_{\lambda',i} \otimes p_{\lambda'',i} \\
      = \sum_{k=0}^n h_{k,i} \otimes h_{n-k,i}. \qedhere
  \end{multline*}
\end{proof}

\begin{lem} \label{lem:ph-adjoint-action}
  For all $n,k \in \N$ and $i,j \in I$, we have
  \[
    p_{k,i}^*(h_{n,j}) = \frac{[k \langle i, j \rangle]}{k} h_{n-k,j}.
  \]
\end{lem}

\begin{proof}
  By~\eqref{eq:h-def} and Proposition~\ref{prop:q-Heis-form-nondegen}, we have
  \begin{multline*}
    p_{k,i}^*(h_{n,j}) = \sum_{\lambda \vdash n} \frac{1}{Z_\lambda} p_{k,i}^*(p_{\lambda,j})
    = \sum_{\lambda \vdash n} \frac{1}{Z_\lambda} m_k(\lambda) [k\langle i,j \rangle] \frac{[k]}{k} p_{\lambda \ominus k, j} \\
    = \sum_{\lambda \vdash n} \frac{1}{Z_{\lambda \ominus k}} \frac{[k \langle i,j \rangle]}{k} p_{\lambda \ominus k, j}
    = \frac{[k \langle i,j \rangle]}{k} h_{n-k,j}. \qedhere
  \end{multline*}
\end{proof}

\begin{lem} \label{lem:hh-adjoint-action}
  For all $k, n \in \N$ and $i,j \in I$, we have
  \begin{gather*}
    h_{k,i}^*(h_{n,j}) = [k+1] h_{n-k,j},\quad \text{when } \langle i,j \rangle = 2, \\
    h_{k,i}^*(h_{n,j}) =
    \begin{cases}
      h_{n-k,i} & \text{if } k=0,1, \\
      0 & \text{if } k > 1,
    \end{cases}
    \quad \text{when } \langle i,j \rangle = -1, \\
    h_{k,i}^*(h_{n,j}) =
    \begin{cases}
      h_{n,j} & \text{if } k=0, \\
      0 & \text{if } k > 0,
    \end{cases}
    \quad \text{when } \langle i,j \rangle = 0.
  \end{gather*}
\end{lem}

\begin{proof}
  We prove the result by induction on $k$.  The result is clear for $k=0$.  Now, for $r \in \N_+$, note that
  \[
    \frac{[r \langle i,j \rangle]}{[r]} =
    \begin{cases}
      \frac{q^{-r\langle i,j \rangle} - q^{r\langle i,j \rangle}}{q^{-r} - q^r} & \text{if } \langle i,j \rangle \ge 0, \\
      (-1)^{r\langle i,j \rangle+1} \, \frac{q^{-r\langle i,j \rangle} - q^{r\langle i,j \rangle}}{q^{-r} - q^r} & \text{if } \langle i,j \rangle < 0.
    \end{cases}
  \]
  For $k > 0$, we have, by~\eqref{eq:h-hp-expansion} and Lemma~\ref{lem:ph-adjoint-action},
  \begin{align*}
    h_{k,i}^* (h_{n,j}) &= \frac{1}{k} \sum_{r=1}^k \frac{r}{[r]} h_{k-r,i}^* p_{r,i}^* (h_{n,j}) \\
    &= \frac{1}{k} \sum_{r=1}^k \frac{[r\langle i,j \rangle]}{[r]} h_{k-r,i}^*(h_{n-r,j}).
  \end{align*}
  The case $\langle i, j \rangle=0$ follows immediately.  If $\langle i,j \rangle = 2$, then
  \begin{align*}
    h_{k,i}^* (h_{n,j}) &= \frac{1}{k} \sum_{r=1}^k (q^{-r} + q^r) h_{k-r,i}^*(h_{n-r,j}) \\
    &= \frac{1}{k} \sum_{r=1}^k (q^{-r} + q^r) [k-r+1] h_{n-k,j} \\
    &= \frac{1}{k} \sum_{r=1}^k (q^{-r} + q^r) \frac{q^{-(k-r+1)} - q^{k-r+1}}{q^{-1}-q} h_{n-k,j} \\
    &= \frac{1}{k}\frac{1}{q^{-1}-q} \sum_{r=1}^k \left( q^{-k-1} + q^{-k+2r-1} - q^{k-2r+1} - q^{k+1} \right) h_{n-k,j} \\
    &= \frac{q^{-k-1} - q^{k+1}}{q^{-1} - q} h_{n-k,j} = [k+1] h_{n-k,j},
  \end{align*}
  where we have used the inductive hypothesis in the second equality.  If $\langle i,j \rangle = -1$, then
  \[
    h_{1,i}^*(h_{n,j}) = h_{0,i}^*(h_{n-1,j}) = h_{n-1,j},
  \]
  and, for $k > 1$,
  \[
    h_{k,i}^* (h_{n,j}) = \frac{1}{k} \sum_{r=1}^k (-1)^{-r} h_{k-r,i}^*(h_{n-r,j}) = 0,
  \]
  where we have used the inductive hypothesis in the second equality.
\end{proof}

For $n \in \N$ and $i \in I$, we let $h_{n,i}$ denote the usual element of $H^+$ defined by~\eqref{eq:h-def} and let $h_{n,i}'$ denote the same element of $\Sy^{\otimes I}$, but considered as an element of $H^-$.

\begin{prop} \label{prop:q-Heis-relations}
  If $A$ is a Cartan matrix of finite or affine ADE type other than type $A_2^{(1)}$, then the Heisenberg double $\fh(H^+,H^-)$ is generated by $h_{n,i}, h_{n,i}'$, $n \in \N$, $i \in I$, with relations
  \begin{gather}
    h_{n,i} h_{m,j} = h_{m,j} h_{n,i} \quad \text{for all } n,m \in \N,\ i,j \in I, \label{eq:q-Heis-hh-relation} \\
    h_{n,i}' h_{m,j}' = h_{m,j}' h_{n,i}' \quad \text{for all } n,m \in \N,\ i,j \in I, \\
    h_{n,i}' h_{m,i} = \sum_{k \ge 0} [k+1] h_{m-k,i} h_{n-k,i}' \quad \text{for all } n,m \in \N,\ i \in I, \\
    h_{n,i}' h_{m,j} = h_{m,j} h_{n,i}' + h_{m-1,j} h_{n-1,i}' \quad \text{for all } n,m \in \N,\ i,j \in I,\ \langle i, j \rangle = -1, \label{eq:q-Heis-ij-connected} \\
    h_{n,i}' h_{m,j} = h_{m,j} h_{n,i}' \quad \text{for all } n,m \in \N,\ i,j \in I,\ \langle i,j \rangle = 0. \label{eq:q-Heis-ij-not-connected}
  \end{gather}
\end{prop}

\begin{proof}
  The proof of the fact that $\fh = \fh(H^+,H^-)$ is generated by $h_{n,i}, h_{n,i}'$, $n \in \N$, $i \in I$, is analogous to the proof that the ring of symmetric functions is generated by the complete symmetric functions and will be omitted.  A complete set of relations for any twisted Heisenberg double is given by the relations in $H^+$, the relations in $H^-$, and the commutation relations between the elements of a generating set for $H^+$ and a generating set for $H^-$.  It thus suffices to compute the commutation relation between $h_{n,i}'$ and $h_{m,j}$ for $n,m \in \N$ and $i,j \in I$.  For $n,m \in \N$ and $i \in I$, we have, by~\eqref{eq:smash-product}, Lemma~\ref{lem:h-coproduct}, and Lemma~\ref{lem:hh-adjoint-action},
  \[
    h_{n,i}' h_{m,i} = \sum_{k \ge 0} h_{k,i}^*(h_{m,i}) h_{n-k,i}'
    = \sum_{k \ge 0} [k+1] h_{m-k,i} h_{n-k,i}'.
  \]
  The proofs of the remaining relations~\eqref{eq:q-Heis-ij-connected} and~\eqref{eq:q-Heis-ij-not-connected} are analogous.
\end{proof}

\begin{rem}
  In the case that $A$ is a Cartan matrix of affine ADE type other than type $A_2^{(1)}$, Proposition~\ref{prop:q-Heis-relations} recovers the presentation of the quantum toroidal Heisenberg algebra given in~\cite[\S2.2]{CL12} after one identifies $h_{n,i}$ and $h_{n,i}'$ with the generators $p_i^{(n)}$ and $q_i^{(n)}$ (respectively) of \cite[\S2.2]{CL12}.  One could also use the above methods to recover the relations of~\cite[\S2.2.2]{CL12}.  The generators $p_i^{(1^n)}$ and $q_i^{(1^n)}$ of that reference correspond to $q$-deformed analogues of the $n$-th elementary symmetric functions in $H^+$ and $H^-$, respectively.
\end{rem}

\begin{rem} \label{rem:q-Heis-integral-form}
  Note that all of the coefficients in the relations in Proposition~\ref{prop:q-Heis-relations} lie in $\Z[q,q^{-1}]$.  Thus, one could consider the $\Z[q,q^{-1}]$-algebra with generators $h_{n,i}, h_{n,i}'$, $n \in \N$, $i \in I$ and relations~\eqref{eq:q-Heis-hh-relation}--\eqref{eq:q-Heis-ij-not-connected}.  This is an integral form of the quantum toroidal Heisenberg algebra.
\end{rem}

%
\section{Lattice Heisenberg algebras} \label{sec:lattice-Heis}
%

Suppose $L$ is a finitely-generated free $\Z$-module with a symmetric bilinear form $\langle -, - \rangle_L \colon L \times L \to \Z$.  Fix a basis $v_1,\dotsc,v_\ell$ of $L$.

\begin{defin}[Lattice Heisenberg algebra] \label{def:lattice-Heis}
  The \emph{lattice Heisenberg algebra} $\fh^L$ associated to $L$ is the unital $\Q$-algebra with generators $p_{i,n}$, $i \in \{1,\dotsc,\ell\}$, $n \in \Z \setminus \{0\}$, and relations
  \[
    p_{i,n} p_{j,m} = p_{j,m} p_{i,n} + n \delta_{n,-m} \langle v_i, v_j \rangle_L,\quad i,j \in \{1,\dotsc,\ell\},\ n,m \in \Z \setminus \{0\}.
  \]
  Note that when $L=\Z$ and $\langle n, m \rangle_L = nm$, then $\fh^L$ is the usual infinite-dimensional Heisenberg algebra $\fh_\textup{classical}$.
\end{defin}

Let $I = \{1,\dotsc,\ell\}$ and define $\langle -, - \rangle \colon I \times I \to \Z$ by $\langle i,j \rangle = \langle v_i, v_j \rangle_L$ for $i,j \in I$.  Then define $H^+$, $p_{n,i}$, $p_{\lambda,i}$, and $p_\blambda$ for $n \in \N$, $i \in I$, $\lambda \in \cP$, $\blambda \in \cP^I$, as in Section~\ref{sec:q-Heis} (except that we work over $\Q$ instead of $\Q(q)$).   We define a symmetric bilinear form on $H^+$ by
\begin{equation}  \label{eq:lattice-Heis-bilinear-form}
  \langle p_\blambda, p_\bmu \rangle = \delta_{\ell(\underline{\blambda}), \ell(\underline{\bmu})} \sum_{\sigma \in S_{\ell(\underline{\blambda})}} \prod_{r=1}^{\ell(\underline{\blambda})} \delta_{\prt \left( \underline{\blambda}_r \right), \prt \left( \underline{\bmu}_{\sigma(r)} \right)} \prt \left( \underline{\blambda}_r \right) \left\langle \clr\left(\underline{\blambda}_r\right), \clr\left(\underline{\bmu}_{\sigma(r)}\right) \right\rangle.
\end{equation}
Set $H^- = H^+$ and view $\langle -, - \rangle$ as a form $H^- \otimes H^+ \to \kk$.

\begin{lem} \label{lem:lattice-Heis-double-p-gens}
  The Heisenberg double $\fh(H^+,H^-)$ is generated by
  \begin{gather*}
    p_{n,i} = p_{n,i} \in H^+,\ p_{n,i}' = p_{n,i} \in H^-,\quad n \in \N,\ i=1,\dotsc,\ell.
  \end{gather*}
  The relations amongst these generators are
  \[
    p_{m,i} p_{n,j} = p_{n,j} p_{m,i},\quad p_{m,i}' p_{n,j}' = p_{n,j}' p_{m,i}',\quad p'_{m,i} p_{n,j} = p_{n,j} p'_{m,i} + n \delta_{m,n} \langle v_i, v_j \rangle_L.
  \]
\end{lem}

\begin{proof}
  The proof of this result is analogous to that of Proposition~\ref{prop:q-Heis-p-presentation} and thus will be omitted.
\end{proof}

\begin{prop}
  We have the following isomorphisms of $\Q$-algebras.
  \begin{enumerate}
    \item If the bilinear form $\langle -, - \rangle_L$ is nondegenerate, then $\fh^L \cong \fh_\textup{classical}$.
    \item If the bilinear form $\langle -, - \rangle_L$ is identically zero, then $\fh^L \cong \Sy$.
    \item If the bilinear form $\langle -, - \rangle_L$ is degenerate but not identically zero, then $\fh^L \cong \fh_\textup{classical} \otimes_\Q \Sy$.
  \end{enumerate}
\end{prop}

\begin{proof}
  Let $L'$ be the radical of the form $\langle -, - \rangle_L$ and choose a complement $L''$. Thus, we have $L = L' \oplus L''$, with $\langle L', L'' \rangle = 0$.  It follows immediately from Definition~\ref{def:lattice-Heis} that $\fh^L \cong \fh^{L'} \otimes_\Q \fh^{L''}$.  Since the form $\langle -, - \rangle_L$ is trivial when restricted to $L'$, we have $\fh^{L'} \cong \Sy$ as $\Q$-algebras.  On the other hand, the form is nondegenerate when restricted to $L''$.  Thus, by Lemma~\ref{lem:lattice-Heis-double-p-gens}, we have $\fh^{L''} \cong \fh_\textup{classical}$ as $\Q$-algebras.
\end{proof}

\begin{rem}
  By considering the complete and/or elementary symmetric functions in $H^\pm$, one can compute other presentations of $\fh(H^+,H^-)$.  Since the complete and elementary symmetric functions generate the ring of symmetric functions over $\Z$, one obtains in this way integral forms of the lattice Heisenberg algebra (see Remark~\ref{rem:q-Heis-integral-form} and \cite[\S6.2]{SY13}).
\end{rem}

\begin{rem}
  Heisenberg algebras appear in many guises in the literature.  We expect that most such algebras (for example, the (quantum) Heisenberg algebras discussed in \cite{FJW00a,FJW00b,FJW02}) can be described in terms of the Heisenberg double.  In each case, one should be able to recover the presentation in question by choosing an appropriate bilinear form on $\Sy^{\otimes \ell}$, for some $\ell \in \N_+$, and generators for $H^{\pm} = \Sy^{\otimes \ell}$.
\end{rem}


\bibliographystyle{alpha}
\bibliography{RossoSavage-biblist}

\end{document}